\documentclass[11pt,a4paper,twoside]{article}
\usepackage[T1]{fontenc}

\usepackage{enumerate}
\usepackage{amsthm,amsmath,amssymb,graphics}
\usepackage{graphicx}
\usepackage{psfrag}
\newcommand{\field}[1]{\mathbb{#1}}
\newcommand{\R}{\field{R}}

\newcommand{\Hp}{\field{H}}
\newcommand{\N}{\field{N}}
\newcommand{\C}{\field{C}}

\newcommand{\Z}{\field{Z}}

\newcommand{\eps}{\varepsilon}

\newtheorem{theorem}{Theorem}

\newtheorem{proposition}{Proposition}[section]
\newtheorem{theoremnum}[proposition]{Theorem}
\newtheorem{claim}[proposition]{Claim}

\newtheorem{lemma}[proposition]{Lemma}

\title{Strong convergence of Kleinian groups: the cracked eggshell}
\author{James W. Anderson \\ {\footnotesize School of Mathematics, University of Southampton, Southampton SO17 1BJ  England}   \\ \\ Cyril Lecuire \\ {\footnotesize Institut de Math\'ematiques de Toulouse, 31062 Toulouse France} }

\begin{document}

\maketitle

\begin{abstract}
In this paper we give a complete description of the set ${\rm SH}(\pi_1(M))$ of discrete faithful representations of the fundamental group of a compact, orientable, hyperbolizable $3$-manifold with incompressible boundary, equipped with the strong topology, with the description given in term of the end invariants of the quotient manifolds. As part of this description, we introduce coordinates on ${\rm SH}(\pi_1(M))$ that extend the usual Ahlfors-Bers coordinates. We use these coordinates to show the local connectivity of ${\rm SH}(\pi_1(M))$ and study the action of the modular group of $M$ on ${\rm SH}(\pi_1(M))$.
\end{abstract}

\section{Introduction and statement of results}
\label{introduction}

Kleinian groups have been studied since the late $19^{\rm th}$ century in the work of Poincar\'{e} and Fricke and Klein, and more extensively since the work of Ahlfors and Bers in the 1960s and Thurston in the 1970s and 1980s.  In this paper, we consider a particular aspect of the basic question of understanding the behavior of sequences of Kleinian groups.  

There are two standard notions of convergence for a sequence of Kleinian groups.  The first is algebraic convergence, which is convergence on generators.  The second is geometric convergence, which is convergence of the quotient hyperbolic $3$-manifolds.  The topology of algebraic convergence on the set of isomorphic, finitely generated Kleinian groups is roughly understood.  In particular, there is a correspondence between connected components of this set of Kleinian groups and marked homotopy classes of compact $3$-manifolds, up to a natural equivalence.    Geometric convergence is considerably less well-behaved, as a sequence of isomorphic, finitely generated Kleinian groups may converge geometrically to an infinitely generated Kleinian group.

Strong convergence of Kleinian groups combines these two different notions, so that a sequence of Kleinian groups converges strongly to a Kleinian group if it converges to that Kleinian group both algebraically and geometrically.  Thurston proposed a picture of what the space of Kleinian groups with the strong topology looks like, which we consider here.

Let $M$ be a compact, orientable, hyperbolizable $3$-manifold with incompressible boundary.  Let ${\rm D}(\pi_1(M))$ denote the set of discrete faithful representations $\rho:\pi_1(M)\rightarrow {\rm PSL}(2,\C)$ equipped with the topology of algebraic convergence (formal definitions are given in Section \ref{background}). The group  ${\rm Isom}^+(\Hp^3)\cong {\rm PSL}(2,\C)$ of orientation-preserving isometries of hyperbolic 3-space acts on ${\rm D}(\pi_1(M))$ by conjugation.   We denote the quotient by ${\rm AH}(\pi_1(M))$.  Note that a sequence $\{ [\rho_n] \}\subset {\rm AH}(\pi_1(M))$ converges to $[\rho]\in {\rm AH}(\pi_1(M))$ if there is a sequence $\{ h_n\}\subset {\rm PSL}(2,\C)$ so that $\{ h_n \rho_n h_n^{-1}\}$ converges to $\rho$ in ${\rm D}(\pi_1(M))$.  When there is no possibility of confusion, we refer to elements of both ${\rm D}(\pi_1(M))$ and ${\rm AH}(\pi_1(M))$ as representations, even though the elements of the latter set are formally equivalence classes of representations.

By considering the quotient manifold $M_\rho = \Hp^3/\rho(\pi_1(M))$ corresponding to a representation $\rho$, we can identify ${\rm AH}(\pi_1(M))$ with the set of hyperbolic $3$-manifolds homotopy equivalent to $M$ up to isometry, where the homotopy equivalence induces  the given representation. A sequence $\{\rho_n\}$ in ${\rm AH}(\pi_1(M))$ converges strongly to a representation $\rho_\infty$ if $\{\rho_n\}$ converges to $\rho_\infty$ algebraically and if $\{\rho_n(\pi_1(M))\}$ converges geometrically to $\rho_\infty(\pi_1(M))$. We  denote by ${\rm SH}(\pi_1(M))$ the set ${\rm AH}(\pi_1(M))$ equipped with the strong topology.  Though we do not use this description, one can identify ${\rm SH}(\pi_1(M))$ with the space of hyperbolic manifolds  homotopy equivalent to $M$ (up to isometry) equipped with a marked pointed Hausdorff-Gromov topology.  

By Bonahon, see \cite{bouts}, each representation $\rho\in{\rm AH}(\pi_1(M))$ is tame, namely the corresponding quotient manifold $M_\rho = \Hp^3/\rho(\pi_1(M))$ is homeomorphic to the interior of a compact $3$-manifold. Combining this result with the uniqueness of the compact core, see \cite{mms}, it follows that the topology of the quotient manifold does not change under strong convergence (compare with \cite{canmi}).  Namely if $\{\rho_n\}$ converges strongly to $\rho_\infty$, then $M_\infty=\Hp^3/\rho_{\infty}(\pi_1(M))$ is homeomorphic to $M_n=\Hp^3/\rho_n (\pi_1(M))$ for $n$ sufficiently large. Using the Ending Lamination Theorem, see \cite{elc1} and \cite{elc2}, and Thurston's Double Limit Theorem, see \cite{thuii}, one can then see that the connected components of ${\rm SH}(\pi_1(M))$ are in one-to-one correspondence with the set of marked homeomorphism types of $3$-manifolds homotopy equivalent to $M$ (we  see that this fact can be deduced from our results as well). In particular, as a first point of contrast with ${\rm AH}(\pi_1(M))$, we show that ${\rm SH}(\pi_1(M))$ does not have the bumping phenomenon described for ${\rm AH}(\pi_1(M))$ for many $M$ in \cite{cancan}. 

In addition to the existence of bumping, another disturbing property of ${\rm AH}(\pi_1(M))$ is that it may not even be locally connected, see \cite{bromcon}, \cite{magid} and \cite{bbcm}. We will show that this does not happen with the strong topology.

\begin{theorem}         \label{locon}
Let $M$ be a compact, orientable, hyperbolizable $3$-manifold with incompressible boundary. Then the space ${\rm SH}(\pi_1(M))$ is locally connected.
\end{theorem}

Even when $M$ is such that ${\rm AH}(\pi_1(M))$ is locally connected, the space ${\rm AH}(\pi_1(M))$ may have so-called self-bumping points, see for instance \cite{mcm}, \cite{brombholt} and \cite{ohshika-div}. We will see that this kind of phenomenon does not appear in ${\rm SH}(\pi_1(M))$.

\begin{theorem} \label{selfb}
Let $M$ be a compact, orientable, hyperbolizable $3$-manifold with incompressible boundary.
Let $\rho\in {\rm SH}(\pi_1(M))$ be a representation uniformizing $M$. Then every neighborhood of $\rho$ contains a neighborhood ${\cal V}\subset {\rm SH}(\pi_1(M))$ of $\rho$ such that ${\cal V}\cap {\rm int}({\rm SH}(\pi_1(M)))$ is connected.
\end{theorem}

The reason why self-bumping points may be present in ${\rm AH}(\pi_1(M))$ and not in ${\rm SH}(\pi_1(M))$ is that some sequences of representations that converge in ${\rm AH}(\pi_1(M))$ do not converge in ${\rm SH}(\pi_1(M))$. This has the following consequence:
 
\begin{lemma}           \label{locom}
Let $M$ be a compact, orientable, hyperbolizable $3$-manifold with incompressible boundary. Then the space ${\rm SH}(\pi(M))$ is not locally compact.
\end{lemma}

We denote by ${\rm AH}(M)\subset {\rm AH}(\pi_1(M))$ the set of representations $\rho$ whose quotient manifold $\Hp^3/\rho(\pi_1(M))$ is homeomorphic to the interior of $M$ by a homeomorphism that induces $\rho$.  This set is contained in the closure of the component of the interior of ${\rm AH}(\pi_1(M))$ containing the minimally parabolic Kleinian groups $\Gamma$ for which $M_\Gamma$ is homeomorphic to the interior of $M$ (this can be deduced from \cite{marden}, \cite{sulli}, \cite{elc2}).  Let ${\rm SH}(M)$ denote the set ${\rm AH}(M)$ with the strong topology. Let ${\rm Mod}(M)$ be the group of isotopy classes of orientation-preserving diffeomorphisms of $M$. When $M = S\times I$ is a trivial $I$-bundle, Kerckhoff and Thurston, see \cite{kt}, have shown that the action of ${\rm Mod}(M) = MCG(S)$ on ${\rm AH}(M)$ is not properly discontinuous.  This result has been extended to other manifolds by Canary and Storm, see \cite{canary-storm}.  On ${\rm SH}(M)$ we have the following result.

\begin{theorem} \label{modul}
Let $M$ be a compact, orientable, hyperbolizable $3$-manifold with incompressible boundary. Assume that $M$ is not an $I$-bundle over a closed surface. Then the action of ${\rm Mod}(M)$ on ${\rm SH}(M)$ is properly discontinuous.
\end{theorem}

Furthermore, we show that when $M$ is a trivial $I$-bundle, this action has fixed points and hence is not properly discontinuous.

The authors would like to thank the referee for their careful reading of the paper.

\subsection{Outline of the paper}
\label{outline}

Thurston \cite{thui} gives the following conjectural description of ${\rm AH}(M)$ and ${\rm SH}(M)$. When $M$ does not contain an essential annulus, ${\rm AH}(M)$ is homeomorphic to a closed ball, like an hard-boiled egg with its shell, whereas ${\rm SH}(M)$ is obtained by thoroughly cracking the eggshell on a hard surface. In order to get a precise description of this cracked eggshell,  we  introduce some coordinates for ${\rm SH}(M)$.

Say that a representation $\rho\in {\rm AH}(M)$ {\em uniformizes} $M$, and note that by definition, $M_\rho$ has a compact core $K$ homeomorphic to $M$.  Without loss of generality, we may choose $K$ to intersect each cusp of $M_\rho$ in either a single annulus (in the case of a rank $1$ cusp) or a torus (in the case of a rank $2$ cusp), see \cite{mcc}.  Define the ends of $M_\rho$ to be the complementary regions of $K\cup cusps$. Such an end is {\em geometrically infinite} if it can be chosen to be contained in the convex core $C_\rho$ of $M_\rho$ and {\it geometrically finite} otherwise. We  define the {\em conformal end invariants} of $\rho$ to consist of the hyperbolic metrics associated to the conformal structure at infinity of each geometrically finite end and of the ending laminations associated to its geometrically infinite ones. Such an invariant $(F,m,L)$ consists of a complete hyperbolic metric $m$ on an open subsurface $F$ of $\partial M$ and of a geodesic lamination $L\subset\partial M-F$ on its complement; we  call it a {\em gallimaufry} on $M$ or on $\partial M$ (see definition in Section \ref{gall}). We put on the set of gallimaufries a topology that extends the usual topology of the Teichm\"uller space. We  say that the gallimaufry $(F,m,L)$ is {\em doubly incompressible} if there is a transverse measure $\lambda$ supported by $L$ and $\eta>0$ such that for every essential disc, annulus or M\"obius band $(E,\partial E)\subset (M,\partial M-F)$ we have $i(\partial E,\lambda)\geq\eta$.   We prove the following result:
 
  \begin{theorem}        \label{conf}
 Let $M$ be a compact, orientable, hyperbolizable $3$-manifold with incompressible boundary.
 The ending map that associates to a representation uniformising $M$ its  end invariants is a homeomorphism from ${\rm SH}(M)$ into the set of doubly incompressible gallimaufries on $\partial M$.
 \end{theorem}

By the proof of the Ending Lamination Theorem, see \cite{elc1} and \cite{elc2}, the map associating its ending gallimaufry to a (conjugacy class of) representation is one-to-one. Thus we  only need to show that the ending map defined in Theorem \ref{conf} is proper and continuous.
 
 We  then prove that the set of doubly incompressible gallimaufries equipped with its topology is locally connected, does not have self-bumping points, and is not locally compact.

In Section \ref{background}, we  give the definitions and some basic results that are used in the paper. In particular, we  define the space of gallimaufries and describe some of its properties. The proof that the ending map is proper  is divided in two. In Section \ref{algebric} we  show that given a sequence of representations whose end invariants converge to a doubly incompressible gallimaufry, a subsequence converges algebraically. This can be viewed as a refinement of Thurston's Double Limit Theorem. The proof  mixes arguments of Thurston's original proof, see \cite{thuii}, arguments of Otal's proof, see \cite{meister2}, and a cut and paste operation. In Section \ref{strong} we  show that the algebraically convergent subsequence provided in the preceding Section actually converges strongly, by studying the behaviour of the convex cores of the sequence and concluding that the limit sets converge in the Hausdorff topology. In Section \ref{continuity} we  show that the ending map is continuous. The main difficulty here is in handling the geometrically infinite ends of the limit. This will be dealt with by proving a relative version of Bonahon's Intersection Lemma, see \cite{bouts}. At this point we  have proved Theorem \ref{conf} and Lemma \ref{locom}. In Section \ref{secselb}, we  prove Theorems \ref{locon} and \ref{selfb} by constructing appropriate paths in the space of doubly incompressible gallimaufries. In Section \ref{secmodul} we  prove Theorem \ref{modul}.

\section{Background material and definitions}
\label{background}

The purpose of this Section is to provide the background material we use in this paper.    We often pass to subsequences; unless otherwise stated, we will without further comment use the same notation for the subsequence as for the original sequence.

\subsection{Kleinian groups and $3$-manifolds}
\label{kleinian}

Standard sources for material on Kleinian groups and hyperbolic $3$-manifolds are \cite{notes}, \cite{jap} and \cite{kapo}.   A {\em Kleinian group} is a discrete subgroup of the group ${\rm PSL}(2,\C)$ of all orientation-preserving isometries of hyperbolic $3$-space $\Hp^3$.  Throughout this paper, we assume that Kleinian groups are torsion-free and non-elementary, so that they contain a non-abelian free subgroup.  

A {\em hyperbolizable $3$-manifold} is an orientable $3$-manifold that admits a complete Riemannian metric all of whose sectional curvatures are equal to $-1$. A $3$-manifold endowed with such a metric is a {\em hyperbolic $3$-manifold}. It follows that an orientable hyperbolic $3$-manifold $N$ can be expressed as the quotient $N =\Hp^3/\Gamma$ for a Kleinian group $\Gamma$, and that $\Gamma$ is unique up to conjugacy in ${\rm PSL}(2,\C)$.

An (orientation-preserving) isometry of $\Hp^3$ extends to a conformal homeomorphism on the boundary at infinity $\partial_\infty\Hp^3 = \widehat\C$ of $\Hp^3$. The {\it domain of discontinuity} $\Omega_\Gamma$ of a Kleinian group $\Gamma$ is the largest open subset of $\widehat\C$ on which the action of $\Gamma$ is properly discontinuous. 

The {\it limit set} $\Lambda_\Gamma \subset \widehat\C$ of $\Gamma$ is the complement of $\Omega_\Gamma$. Let $H_\Gamma\subset\Hp^3$ be the {\it convex hull} in $\Hp^3$ of $\Lambda_\Gamma$, which is the smallest non-empty convex subset of $\Hp^3$ invariant under the action of $\Gamma$.  The {\it convex core} $C_\Gamma$ of $\Gamma$ is the quotient of $H_\Gamma$ by the action of $\Gamma$.

A Kleinian group $\Gamma$ is {\em geometrically finite} if the unit neighbourhood of its convex core $C_\Gamma$ has finite volume in $\Hp^3/\Gamma$.  A Kleinian group $\Gamma$ is {\em convex co-compact} if its convex core $C_\Gamma$ is compact, or equivalently, if it is geometrically finite and contains no parabolic elements. A Kleinian group is {\em minimally parabolic} if every parabolic isometry belongs to a rank two free abelian subgroup.

A compact $3$-manifold is {\em hyperbolizable} if there exists a Kleinian group $\Gamma$ so that the interior ${\rm int}(M)$ of $M$ is homeomorphic to the quotient manifold $\Hp^3/\Gamma$.  Note that, under this definition, a hyperbolizable $3$-manifold $M$ is necessarily {\em irreducible} (so that every embedded $2$-sphere in $M$ bounds a $3$-ball in $M$), orientable, and {\em atoroidal} (so that every incompressible torus $T$ in $M$ is homotopic into $\partial M$).

An embedded compact submanifold $W$ in a $3$-manifold $M$ is {\em incompressible} if $\pi_1(W)$ is infinite and if the inclusion $W\hookrightarrow M$ induces an injective map on fundamental groups.  A compact $3$-manifold has {\em incompressible boundary} if each component of $\partial M$ is incompressible.

For a compact $3$-manifold $M$, let $\partial_{\chi<0} M$ denote the union of the components of $\partial M$ of negative Euler characteristic, so that $\partial_{\chi<0} M$ consists of $\partial M$ with all boundary tori removed.

Let $F\subset\partial M$ be an incompressible compact surface. An {\em essential annulus} (or M\"obius band) $E$ in $(M,F)$ is a properly embedded incompressible annulus (or M\"obius band) $(E,\partial E)\subset (M,F)$ that cannot be homotoped into $F$ by a homotopy fixing $\partial E$. An {\em essential disc} $D$ is a properly embedded disc $(D,\partial D)\subset (M,F)$ that cannot be homotoped into $F$ by a homotopy fixing $\partial D$.

\subsection{Discrete faithful representations}
\label{dis}

Let $M$ be a compact, orientable, hyperbolizable $3$-manifold.  A {\em discrete faithful representation} of $\pi_1(M)$ into ${\rm PSL}(2,\C)$ is an injective homomorphism $\rho: \pi_1(M)\rightarrow {\rm PSL}(2,\C)$ whose image $\rho(\pi_1(M))$ is a Kleinian group.  A discrete faithful representation $\rho$ of $\pi_1(M)$ is {\em geometrically finite} if the image group $\rho(\pi_1(M))$ is geometrically finite, and is {\em convex co-compact} if $\rho(\pi_1(M))$ is convex co-compact.  

Let $\rho:\pi_1(M)\rightarrow {\rm PSL}(2,\C)$ be a discrete faithful representation. Assume moreover that the quotient manifold $M_\rho=\Hp^3/\rho(\pi_1(M))$ is homeomorphic to ${\rm int}(M)$ by a homeomorphism $f: {\rm int}(M)\rightarrow M_\rho=\Hp^3/\rho(\pi_1(M))$ that induces $\rho$, so that $\rho = f_*$. Under these assumptions, say that $\rho$ {\em uniformises} $M$.  

Choose now a pairwise disjoint set of horoballs in $\Hp^3$ so that each horoball is centered at a parabolic fixed point of $\rho(\pi_1(M))$, and there is a horoball centered at the fixed point of each parabolic subgroup of $\rho(\pi_1(M))$ and hence invariant under the corresponding parabolic subgroup.  It is a standard application of the Margulis lemma that such a set of horoballs exists.   The {\em cuspidal part} of $M_\rho$ is the quotient of such a collection of horoballs.  Where relevant, we will assume that we have made a convenient choice of such a collection of horoballs. 

Let ${\rm D}(\pi_1(M))$ be the space of all discrete faithful representations of $\pi_1(M)$ into ${\rm PSL}(2,\C)$.  By choosing a fixed set of $p$ generators $g_1,\ldots, g_p$ for $\pi_1(M)$, we can realize ${\rm D}(\pi_1(M))$ as a subspace of $({\rm PSL}(2,\C))^p$ by the identification $\rho\mapsto (\rho(g_1),\ldots, \rho(g_p))$; the topology thus obtained is called the {\em algebraic topology} or the {\em topology of algebraic convergence}. By \cite{jorgi}, ${\rm D}(\pi_1(M))$ is closed in this topology. In this paper, we do not work with all of ${\rm D}(\pi_1(M))$, but rather with the connected components of its interior; for a full description of ${\rm D}(\pi_1(M))$ in light of the proof of the Ending Lamination Theorem, see \cite{accarre}. The representations contained in a given connected component of the interior of ${\rm D}(\pi_1(M))$ are the geometrically finite minimally parabolic representations uniformising a given compact, orientable, hyperbolizable $3$-manifold $M'$ homotopy equivalent to $M$, by work of Ahlfors, Bers, Kra, Marden, Maskit, Sullivan, and Thurston. Such a component of ${\rm int}({\rm D}(\pi_1(M)))$ is uniquely defined by the given manifold $M'$.

If we instead consider the Chabauty topology on subgroups of ${\rm PSL}(2,\C)$, we get the {\em geometric topology} or the {\em topology of geometric convergence}.  A sequence $\{\Phi_n \}$ of Kleinian groups {\em converges geometrically} to a Kleinian group $\Phi_\infty$ if every accumulation point $a$ of every sequence $\{ a_n\in \Phi_n\}$ lies in $\Phi_\infty$ and if every element $a$ of $\Phi_\infty$ is the limit of a sequence $\{ a_n\in \Phi_n\}$.  We will normally consider geometric convergence for the sequence $\{ \rho_n(\pi_1(M))\}$ for a sequence $\{\rho_n\}\subset {\rm D}(\pi_1(M))$; let $G_\infty$ be the geometric limit of  $\{ \rho_n(\pi_1(M))\}$.  In this case, note that the limit manifold $\Hp^3/G_\infty$ is not necessarily homeomorphic to ${\rm int}(M)$, and in fact the geometric limit $G_\infty$ of a sequence of (isomorphic) finitely generated Kleinian groups can be infinitely generated.

The sequence $\{\rho_n\}\subset {\rm D}(\pi_1(M))$ {\em converges strongly} if $\{\rho_n\}$ converges algebraically to some representation $\rho_\infty$ and if $\{ \rho_n(\pi_1(M))\}$ converges geometrically to $\rho_\infty(\pi_1(M))$. As we commented  in the introduction, strong convergence preserves the topology of the quotient manifold.

\begin{lemma}           \label{topo}
Let $M$ be a compact, orientable, hyperbolizable $3$-manifold with incompressible boundary.  Let $\{\rho_n\}$ be a sequence of representations uniformising $M$ that converges strongly to a representation $\rho_\infty$. Then $\rho_\infty$ uniformises $M$.
\end{lemma}

\begin{proof}
By the proof of Marden's tameness conjecture for such $3$-manifolds, see \cite{bouts} (see also \cite{agol} or \cite{caga} for the general case), there is a compact $3$-manifold $M'$ such that $M_\infty=\Hp^3/\rho_\infty (\pi_1(M))$ is homeomorphic to ${\rm int}(M')$. Consider a compact core $K$ for $M_\infty$, which we can choose to be homeomorphic to $M'$. Since $\{\rho_n(\pi_1(M))\}$ converges geometrically to $\rho_\infty(\pi_1(M))$, there are points $x_n$, $n\in\N\cup \{ \infty\}$, and maps $\phi_n:M_\infty\rightarrow M_n$ with $\phi_n(x_\infty)=x_n$ such that the restrictions of the $\phi_n$ to $B(x_\infty,R_n)$ are $q_n$-bilipschitz with $R_n\longrightarrow\infty$ and $q_n\longrightarrow 1$, see \cite{jap}. For $n$ sufficiently large, we have $K\subset B(x_\infty,R_n)$. By construction $\phi_n(K)$ is a compact core for $M_n$. By \cite{mms}, any compact core for $M_n$ is homeomorphic to $K$. It follows that $M_n$ is homeomorphic to $M_\infty$ for $n$ sufficiently large.  
\end{proof}

On ${\rm int}({\rm D}(\pi_1(M)))$, the algebraic and strong topologies are equivalent.

The group ${\rm PSL}(2,\C)$ acts on ${\rm D}(\pi_1(M))$ by conjugation. We denote by ${\rm AH}(\pi_1(M))$ the quotient of ${\rm D}(\pi_1(M))$  by ${\rm PSL}(2,\C)$ endowed with the algebraic topology. We  denote by ${\rm SH}(\pi_1(M))$ the quotient of ${\rm D}(\pi_1(M))$ by ${\rm PSL}(2,\C)$ endowed with the strong topology. The interiors of both ${\rm AH}(\pi_1(M))$ and ${\rm SH}(\pi_1(M))$, with their respective topologies, are the quotient of ${\rm int}({\rm D}(\pi_1(M))$ by ${\rm PSL}(2,\C)$.

Let ${\rm D}(M)$ be the set of representations $\rho\in {\rm D}(\pi_1(M))$ that uniformise $M$. Recall that this means that  the quotient manifold $M_\rho=\Hp^3/\rho(\pi_1(M))$ is homeomorphic to ${\rm int}(M)$ by a homeomorphism $f: {\rm int}(M)\rightarrow M_\rho=\Hp^3/\rho(\pi_1(M))$ that induces $\rho$.  (Hence, the interior of ${\rm D}(M)$ consists of quasiconformal deformations of a given minimally parabolic Kleinian group uniformizing $M$.)  We denote by ${\rm AH}(M)$, respectively ${\rm SH}(M)$, the quotient of ${\rm D}(\pi_1(M))$  by ${\rm PSL}(2,\C)$ endowed with the algebraic topology, respectively the strong topology. When $\partial M$ is incompressible, it follows from the work of Ahlfors, Bers, Kra, Maskit, Sullivan and Thurston that ${\rm int}({\rm AH}(M))={\rm int}({\rm SH}(M))$ is homeomorphic to the Teichm\"uller space ${\cal T}(\partial_{\chi<0} M)$; in particular it is a topological ball.

It follows from Lemma \ref{topo} that for every topological manifold $M$, ${\rm SH}(M)$ is a closed set and is disjoint from any other set ${\rm SH}(M')\subset{\rm SH}(\pi_1(M))$. Hence to understand ${\rm SH}(\pi_1(M))$, we only need to understand each of the ${\rm SH}(M')$ separately as $M'$ varies over all the $3$-manifolds homotopy equivalent to $M$.

\medskip

An $\R$-tree is a path metric space such that every two points can be joined by a unique geodesic arc. Let ${\cal T}$ be an $\R$-tree and let $G$ be a group acting on ${\cal T}$ by isometries.  We say that the action is {\em small} if the edge stabilizers are virtually cyclic.  We say that the action is {\em minimal} if no proper subtree of ${\cal T}$ is invariant under the action of $G$.  We will normally consider actions of $\R$-trees that are both small and minimal.

Morgan and Shalen \cite{mors1} introduced a compactification of ${\rm AH}(M)$ by small minimal isometric actions of $\pi_1(M)$ on $\R$-trees. Consider an element $g$ of $\pi_1(M)$. For $\rho\in {\rm AH}(M)$ we  denote by $\ell_\rho(g)$ the translation distance of $\rho(g)$ (which is $0$ if the isometry is parabolic). Given an isometric action of $\pi_1(M)$ on an $\R$-tree ${\cal T}$, we  denote by $\ell_{\cal T}(g)$ the translation distance of the action of $g$ on ${\cal T}$. A sequence of representations $\{\rho_n\}\subset {\rm AH}(M)$ tends to a minimal isometric action of $\pi_1(M)$ on an $\R$-tree ${\cal T}$ if there are $\eps_n\longrightarrow 0$ such that $\eps_n\ell_{\rho_n}(g)\longrightarrow \ell_{\cal T}(g)$ for every element $g$ of $\pi_1(M)$, see \cite{conti}. Multiplying the sequence $\eps_n$ by a given constant yields an isometric action on another $\R$-tree which is homothetic to the first one. Hence ${\rm AH}(M)$ has thus been compactified by actions on $\R$-trees up to homothety.

Replacing $\Hp^3$ by $\Hp^2$ we get a compactification of the Teichm\"uller space which by Skora's Theorem, see  \cite{skora}, is equivalent to Thurston's compactification by projective measured geodesic laminations (see Section \ref{compactification}).

\subsection{Geodesic laminations}
\label{geod lam}

Standard sources for material on geodesic laminations are \cite{penner} and \cite[Appendice]{meister2}.  Note that unless otherwise explicitly stated, a surface of negative Euler characteristic will be equipped with a complete hyperbolic metric of finite area.

A {\em geodesic lamination} $L$ on a closed hyperbolic surface $S$ is a compact set that is the (non-empty) disjoint union of complete embedded geodesics. Note that this definition can be made independent of the choice of metric on $S$, see \cite[Appendice]{meister2} for example.  A geodesic lamination is {\em minimal} if it does not contain a geodesic lamination as a proper subset. A minimal lamination is either a simple closed geodesic or an {\em irrational lamination}.  A leaf of a geodesic lamination is {\em recurrent} if it lies in a minimal sublamination.  A geodesic lamination is the disjoint union of finitely many minimal sublaminations and finitely many non-recurrent leaves.   The {\em recurrent part} of a geodesic lamination is the union of its recurrent leaves. It is itself a geodesic lamination.  A {\em multi-curve} is a union of disjoint simple closed geodesics.  We  say that two geodesic laminations $L$ and $L'$ {\em cross} if at least one leaf of $L$ transversely intersects a leaf of $L'$.

Let $F$ be a compact surface with boundary. We define a geodesic lamination in $F$ similarly as in the case of a closed surface by considering a hyperbolic metric with geodesic boundary on $F$ (again the definition can be made independent of the metric). A geodesic lamination $L\subset F$ is {\em peripheral} if $L$ is a simple closed curve freely homotopic to a component of $\partial F$.

Let $L\subset S$ be a connected geodesic lamination that is not a simple closed curve. The {\em warren} $W(L)$ of $L$ is the smallest subsurface of $S$ with geodesic boundary containing $L$. Notice that we may have $W(L) =S$.
It is not hard to see that $W(L)$ contains finitely many simple closed curves that are disjoint from $L$, see \cite[\textsection 2.4]{espoir}. Removing from $W(L)$ an annulus around each such curve, we get the {\em surface embraced by} $L$ that we  denote by $S(L)$. When $L$ is a simple closed curve, we take $S(L)$ to be an (open) annular neighbourhood of $L$. In the particular case that $L$ is a non-connected geodesic lamination, we take $S(L)$ to be the disjoint union of the $S(L_i)$ as $L_i$ runs through the connected components of $L$. Notice that $S(L)$ is an open surface (i.e. without boundary).

 A {\em measured geodesic lamination} $\lambda$ consists of a geodesic lamination $|\lambda|$ and a transverse measure on $|{\lambda}|$. Any arc $k\cong [0,1]$ embedded in $S$ transverse to $|\lambda|$, such that  $\partial k\subset S-|\lambda|$, is endowed with a transverse measure $d\lambda$ such that:

 \begin{enumerate}[-]
 \item the support of $d\lambda |_{k}$ is $|\lambda|\cap k$;
 \item if an arc $k'$ can be homotoped to $k$ by a homotopy preserving $|\lambda|$ then $\int_k\! d\lambda=\int_{k'} d\lambda$.
 \end{enumerate}

If $\lambda$ is a measured geodesic lamination, then its support $|\lambda|$ contains only recurrent leaves. Two measured geodesic laminations cross if their supports cross.

We  denote by ${\cal ML}(S)$ the space of measured geodesic laminations on $S$ endowed with the weak$^*$ topology on transverse measures. If $\gamma$ is a weighted simple closed curve with weight $w(\gamma)$ and $\lambda$ is a measured geodesic lamination transverse to $\gamma$, the {\em intersection number} $i(\lambda,\gamma)$ is defined by $i(\lambda,\gamma)=w(\gamma)\int_{|\gamma|}\! d\lambda$. Weighted simple closed curves are dense in ${\cal ML}(S)$ and so $i$ extends to a continuous function $i:{\cal ML}(S)\times {\cal ML}(S)\rightarrow \R$ (\cite{rees}, see also \cite{bouts}).

Given a complete hyperbolic metric $s$ on $S$, the length of a weighted simple closed curve with support $c$ and weight $w\in\R$ is $w \ell_s(c)$, where $\ell_s(c)$ is the length of $c$ with respect to $s$. This length function extends continuously to a function $\ell_s:{\cal ML}(S)\rightarrow \R$ called the {\em length function}, see \cite{bouts}.

This definition can be extended to define the length of a measured geodesic lamination in a hyperbolizable $3$-manifold as follows. Let $M$ be a compact hyperbolizable $3$-manifold with boundary. We are not interested in curves lying in a torus component of $\partial M$, so we use the notation ${\cal ML}(\partial M)$ for ${\cal ML}(\partial_{\chi<0} M)$. Let $S$ be compact subsurface of $\partial_{\chi<0} M$. Let $\rho\in {\rm AH}(M)$ be a representation uniformizing $M$ and let $c\subset S$ be a simple closed curve. Denote by $c^*$ the closed geodesic in $M_{\rho}=\Hp^3/\rho(\pi_1(S))$ in the free homotopy class defined by $c$, if such a geodesic exists. We denote by $\ell_{\rho}(c^*)$ the length of $c^*$ with respect to the hyperbolic metric on $M_{\rho}$.  When $\rho(c)$ is a parabolic isometry, we take $\ell_{\rho}(c^*)=0$. This allows us to define the length of a weighted multi-curve.  Using the density of weighted multi-curves in ${\cal ML}(\partial M)$, we can then define the length $\ell_{\rho}(\lambda^*)$ of a measured geodesic lamination $\lambda\in{\cal ML}(\partial M)$.

We can associate to a measured geodesic lamination $\beta$ on a hyperbolic surface $S$ a small minimal action of $\pi_1(S)$ on an $\R$-tree ${\cal T}_{\beta}$ dual to $\beta$, see \cite{meister2}. If $c$ is a simple closed curve, we denote by $\ell_{\beta}(c)$ the translation distance of an isometry of ${\cal T}_{\beta}$ corresponding to $c$. The action of  $\pi_1(S)$ on ${\cal T}_{\beta}$ we get in this way satisfies $\ell_{\beta}(c)= i(\beta,c)$. Notice that this property completely defines the minimal action of $\pi_1(S)$ on ${\cal T}_{\beta}$, see \cite{meister2}.

 Let $S$ be a connected hyperbolic surface and let $q:\Hp^2\rightarrow S$ be the covering projection. Let $L\subset S$ be a geodesic lamination and let $\pi_1(S)\curvearrowright{\cal T}$ be a minimal action of $\pi_1(S)$ on an  $\R$-tree ${\cal T}$. Then $L$ is {\em realized} in ${\cal T}$ if there is a continuous equivariant map $\Hp^2\rightarrow {\cal T}$ whose restriction to any lift of a leaf of $L$ is injective. When $S$ is a component of the boundary of a $3$-manifold $M$, we extend this definition to actions of $\pi_1(M)$ on $\R$-trees in the following way. Given an action of $\pi_1(M)$ on an $\R$-tree ${\cal T}$, we use the map $i_*:\pi_1(\partial M)\rightarrow\pi_1(M)$ induced by the inclusion to get an action of $\pi_1(\partial M)$ on ${\cal T}$ (which is still small when $\partial M$ is incompressible). By saying that $L$ is realized in ${\cal T}$, we actually mean that $L$ is realized in the minimal tree invariant under the action of $i_*(\pi_1(S))$ on ${\cal T}$, equipped with the (minimal) action of $\pi_1(S)$

A simple closed curve $c$ in a compact (i.e. closed or compact with boundary) surface $S$ is {\em essential} if $c$ does not bound a disc in $S$. A compact subsurface $F\subset S$ is {\em essential} if every simple closed curve $c\subset F$ that bounds a disc in $S$ bounds a disc in $F$.
 
\subsection{Train tracks}
\label{train tracks}

Consider a compact surface $S$ endowed with a complete hyperbolic metric of finite area.  The purpose of this Section is to introduce the notion of a {\em train track}, which is an object on a surface used to combinatorially encode measured geodesic laminations.  Specifically, a  {\em train track} $\tau$ in $S$ is the union of finitely many "rectangles" $b_i$, called the  {\em branches}. In a rectangle $[0,1]\times [0,1]$ we call a segment $\{p\}\times [0,1]\subset [0,1]\times [0,1]$ a {\em vertical} segment and a segment $[0,1]\times \{p\}\subset [0,1]\times [0,1]$ a {\em horizontal} segment. The two extremal vertical segments are the {\em vertical sides}. The branches of a train track satisfy:
\begin{enumerate}[-]
\item  a branch $b_i$ is the image of a rectangle $[0,1]\times[0,1]$ under a smooth map whose restriction to $]0,1[\times[0,1]$ is an embedding;
\item  the union of the double points of a branch is either empty or a non-degenerate vertical segment;
\item  given a pair of branches and the corresponding rectangles and maps, the intersection of the images is either empty or a non-degenerate segment;
\item  given the collection of rectangles and smooth maps producing the branches, every connected component of the union of the images of the vertical sides is a simple arc embedded in $S$.
\end{enumerate}
Furthermore, we will assume that the closure of the complement of a train track $\tau$ contains no components that are either discs, monogons, bigons, punctured discs or annuli.

The images of the vertical segments $\{p\}\times [0,1]$ are the {\em ties}. A maximal connected union of ties lying in different branches is a {\em switch}. A sub-track of a train track $\tau$ is a train track all of whose branches are branches of $\tau$.

A geodesic lamination $L $ is  {\em  carried by a train track} $\tau$ when there is a hyperbolic metric $m$ on $S$ such that the $m$-geodesic lamination $L$ lies in $\tau$ and is transverse to the ties. A geodesic lamination $L $ is {\em minimally carried} by $\tau$ if no proper sub-track of $\tau$ carries $L$. A measured geodesic lamination $\lambda$ is {\em carried} by a train track $\tau$ if its support $|\lambda|$ is carried by $\tau$. Associating to each branch $b$ of $\tau$ the transverse measure $\lambda(b)$ of a tie, we get a {\em weight system} for $\tau$. A weight system is a function $\{{\rm branches}\}\rightarrow\R^+$ satisfying the switch conditions: each switch is two-sided and the sum of the weights of the branches on one side is equal to the sum on the other side. If a train track $\tau$ minimally carries a measured geodesic lamination, there is a bijection between the set of measured geodesic laminations carried by $\tau$ and the weight systems for $\tau$, see \cite[Th. 1.7.12 and 2.7.4]{penner}. 

\subsection{End invariants}

Consider a discrete faithful representation $\rho:\pi_1(M)\rightarrow {\rm PSL}(2,\C)$ that uniformizes $M$. By a result of Scott, the quotient manifold $M_\rho$ has a compact core $K$ homeomorphic to $M$. We can choose $K$ to intersect a chosen collection ${\cal C}$ of cusps of $M_\rho$ (see Section \ref{dis}) in annuli or tori. The {\em ends} of $M_\rho$ are (equivalence classes of) the complementary regions in $M_\rho$ of $K\cup {\cal C}$. To each end we associate an open surface $F\subset\partial K\approx\partial M$, and the end is homeomorphic to $F\times[0,\infty)$, by the Tameness Theorem.

Denote the convex core of $M_\rho$ by $C_\rho$.  An end of $M_\rho$, as defined above, is {\em geometrically infinite} if it is contained in the convex core $C_\rho$ of $M_\rho$ (up to a compact piece) and {\it geometrically finite} otherwise.  We note that a Kleinian group is geometrically finite by the definition given earlier in Section \ref{kleinian} if and only if it has finitely many ends, each of which is geometrically finite by this definition.

The quotient $\Omega_\rho/\rho(\pi_1(M))$  of the domain of discontinuity is a Riemann surface of finite type and adds a natural conformal boundary to the open manifold $M_\rho$. This yields a natural embedding $\Omega_\rho/\rho(\pi_1(M))\hookrightarrow\partial M$ well defined up to homotopy. To each geometrically finite end $E=F\times[0,\infty)$ is associated the component of $\Omega_\rho/\rho(\pi_1(M))$ homeomorphic to $F$. Thus we get a point in the Teichm\"uller space ${\cal T}(F)$; this is the  end invariant associated to this end.

Let $E\approx F\times[0,\infty)$ be a geometrically infinite end of $M_\rho$. Thurston associated to $E$, see \cite{notes}, a minimal geodesic lamination $L$ on $F$ defined as follows. Consider a sequence $\{c_n\}\subset F$ of simple closed curves whose geodesic representatives $c_n^*\subset E \subset M_\rho$ exit every compact subset of $E$.  Make $c_n$ into a geodesic lamination by using the counting measure.  The existence of such a sequence of curves is a non-trivial fact proved in \cite{bouts}. Extract a subsequence such that $\{c_n\}$ converges in ${\cal PML}(F)$ to a projective measured geodesic lamination $[\lambda]$. Then $[\lambda]$ is supported by $L$. Furthermore $L$ does not depend on the choice of $\{ c_n\}$, as long as $\{ c_n^*\}$ exits every compact subset of $E$.

Combining the  end invariants of the geometrically finite ends and the ending laminations of the geometrically infinite ends, we get the  end invariants of $\rho$. In Section \ref{gall}, we  describe more precisely the kind of objects thus obtained.

\subsection{Thurston's compactification}                
\label{compactification}

The Teichm\"uller space ${\cal T}(S)$ of a hyperbolic surface $S$ can be compactified by actions on $\R$-trees, up to homothety. The actions thus obtained are small and minimal.  By Skora's Theorem, see \cite{skora}, such an action is dual to a measured geodesic lamination. Since  ${\cal T}(S)$ is compactified by actions on $\R$-trees up to homothety, we actually get a projective measured geodesic lamination. In order to prove some results in the present paper, we  need a more precise description of the behavior of a sequence $\{ m_n\}\subset {\cal T}(S)$ tending to a projective measured geodesic lamination $[\lambda]$. In this Section we  recall how Thurston compactifies ${\cal T}(S)$ directly by projective measured geodesic laminations (that is, without using $\R$-trees), following the exposition given in \cite[Expos\'e 8]{flp}.

Let $m$ be a complete hyperbolic metric on $S$. Let $P\subset S$ be a multi-curve such that the components of $S-P$ are $3$-holed spheres.  We refer to a $3$-holed sphere as a {\em pair of pants}, and such a multi-curve $P$ is called a {\em pants decomposition} of $S$. Following \cite[\textsection I.2. Expos\'e 8]{flp} and using $P$, we  associate a partial {\em measured foliation}  ${\cal F}_P(m)$ of $S$ to $m$.  (Here, a {\em measured foliation} of a surface $S$ is a singular foliation of $S$ with a transverse measure, where the measure of a curve transverse to the foliation is preserved by homotopies preserving the foliation.  A {\em partial measured foliation} is a measured foliation of a subset of $S$.)  Let $R$ be a pair of pants and let $l_1,l_2,l_3$ be the leaves of $P$ bounding $R$. 

Assume first that  $\ell_m(l_i)\leq \ell_m(l_j)+\ell_m(l_k)$ for every $i\neq j\neq k$. For $i\neq j$, let $k_{i,j}$ be the geodesic segment orthogonal to $l_i$ and $l_j$. Let $T_{i,j}$ be the set of points at a distance at most $\frac{1}{4}(\ell_m(l_i)+\ell_m(l_j)-\ell_m(l_k))$ from $k_{i,j}$. This set is foliated by the curves $\{z\: |\: {\rm d}(k_{i,j},z)\mbox{ is constant}\}$ and the transverse measure is given by the distance between two leaves. Thus in $R$ we have three foliated sets. Notice that by the choice of their widths they don't intersect in ${\rm int}(R)$. Thus we have constructed a partial measured foliation ${\cal F}_P(m)$ of $R$ (see figure \ref{figfeuille}).

 \begin{figure}[hbtp]
\psfrag{a}{Triangular inequality}
\psfrag{b}{$\ell_m(l_1)> \ell_m(l_2)+\ell_m(l_3)$}
\psfrag{c}{$l_1$}
\psfrag{d}{$l_2$}
\psfrag{e}{$l_3$}
\psfrag{f}{$k_{1,2}$}
\psfrag{g}{$k_{2,3}$}
\psfrag{h}{$k_{1,3}$}
\psfrag{i}{$k_{1,1}$}
\psfrag{j}{$T_{1,2}$}
\psfrag{k}{$T_{2,3}$}
\psfrag{l}{$T_{1,3}$}
\psfrag{m}{$T_{1,1}$}
\psfrag{n}{$A$}
\psfrag{o}{$A'$}
\centerline{\includegraphics{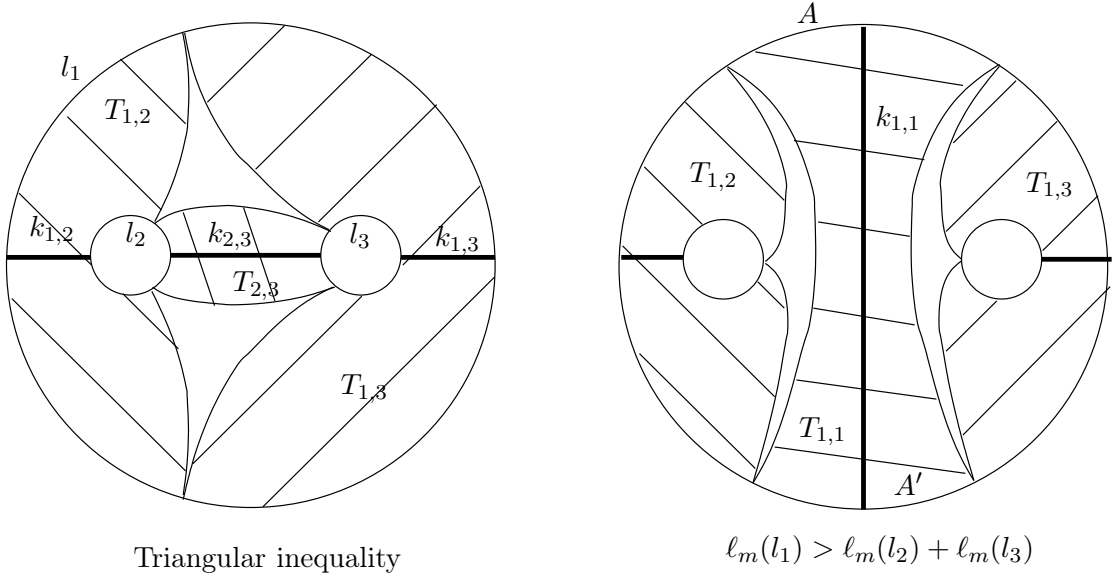}}
\caption{The partial foliation of $R$}
\label{figfeuille}
\end{figure} 

If  on the other hand $\ell_m(l_1)> \ell_m(l_2)+\ell_m(l_3)$ (up to relabeling, this is the only case left to consider) we take $T_{1,2}$, respectively $T_{1,3}$, to be the set of points at a distance at most $\frac{\ell_m(l_2)}{2}$ from $k_{1,2}$, foliated as before, respectively the set of points at a distance at most $\frac{\ell_m(l_3)}{2}$ from $k_{1,3}$. Let $k_{1,1}$ be the (simple) geodesic segment orthogonal at both ends to $l_1$. The set of points of $l_1$ not lying in the interior of $T_{1,2}\cap l_1$ nor of $T_{1,3}\cap l_1$ form two arcs $A$ and $A'$. We take $T_{1,1}$ to be the union of the curves $\{z\: |\: {\rm d}(k_{1,1},z)\mbox{ is constant}\}$ starting at $A$ (notice that $k_{1,1}$ may not lie in $T_{1,1}$). Thus we have constructed a partial measured foliation ${\cal F}_P(m)$ of $R$ (see figure \ref{figfeuille}).

It is proved in \cite[Expos\'e 8]{flp} that the projection mapping $m$ to ${\cal F}_P(m)$ is a homeomorphism from ${\cal T}(S)$ into the set of partial measured foliations giving a non-zero measure to every leaf of $P$. Notice that ${\cal F}_P(m)$ depends on the choice of $P$.

One can extend this measured partial foliation into a singular measured foliation of the whole surface. This singular foliation is well defined up to isotopy. It has been noticed by Thurston that there is a natural bijection between equivalence classes (up to isotopy and Whitehead moves, see \cite{flp}) of singular measured foliations of a surface and measured geodesic laminations, see \cite{levitt}. To a measured foliation ${\cal F}_P(m)$, there corresponds the measured geodesic lamination $\lambda$ satisfying $i({\cal F}_P(m),c)=i(\lambda,c)$ for every simple closed curve $c$. From now on we  give the same name to a foliation and to the corresponding lamination.

The measured geodesic lamination ${\cal F}_P(m)$ roughly describes the length spectrum of $m$, as can be seen in the following result, see \cite[Theorem 2.2]{thuii}:
\begin{theoremnum}           \label{compactlam}
Let $m_0$ be a complete hyperbolic metric on $S$, and let $P$ be a pants decomposition of $S$.  Let $\eps>0$ be such that $\ell_{m_0}(c)\geq\eps$ for every leaf $c$ of $P$. Then there is $Q = Q(\eps)$ such that every measured geodesic lamination $\lambda$ on $S$ satisfies $i(\lambda,{\cal F}_P(m_0))\leq \ell_{m_0}(\lambda)\leq i(\lambda,{\cal F}_P(m_0))+Q \ell_{m_0}(\lambda)$. 
\end{theoremnum}

Now we explain how a projective measured geodesic lamination is associated to a diverging sequence of points of ${\cal T}(S)$. First, note that there exists $\eps>0$ so that if a simple closed geodesic $c\subset S$ has length less than $\eps$ then every simple closed geodesic crossing $c$ has length at least $\eps$. The existence of such an $\eps$ can be easily deduced from the Collar Lemma and $\eps$ does not depend on the hyperbolic metric on $S$. Take any pants decomposition $Q$ of $S$. Given a sequence $\{m_n\}\subset{\cal T}(S)$, we extract a subsequence such that there exists $N$ such that for every leaf $c$ of $Q$ either $\ell_{m_n}(c)\leq\eps$ for every $n\geq N$ or $\ell_{m_n}(c)\geq\eps$ for every $n\geq N$. If a leaf $c$ of $Q$ satisfies $\ell_{m_n}(c)\leq\eps$ for every $n\geq N$ then we replace it by a simple closed curve $d$ that crosses $c$ and is disjoint from $Q-c$. Thus we get a new pants decomposition $P$ such that for every leaf $c$ of $P$, $\ell_{m_n}(c)\geq\eps$ for every $n\geq N$. Now we consider ${\cal F}_P(m_n)$. Assume that $\{m_n\}$ is a diverging sequence. By Theorem \ref{compactlam}, ${\cal F}_P(m_n)$ diverges as well. We extract a subsequence so that $\{ [{\cal F}_P(m_n)]\}$ converges in the space of projective measured geodesic laminations.

Thus we get Thurston's compactification of Teichm\"uller space by projective measured geodesic laminations. Using Theorem \ref{compactlam} it is easy to see that this compactification is the same as the one obtained by combining the Morgan-Shalen compactification by actions on $\R$-trees with Skora's Theorem.

Forgetting the transverse measure, we  say that a sequence $\{m_n\}$ of complete hyperbolic metrics on $S$ {\em tends to a lamination $L$} on $S$ if every subsequence contains a further subsequence converging to a projective measured geodesic lamination supported by $L$.

\subsection{Compactification of Teichm\"uller space for surfaces with boundary}   
\label{withboundary}

Given a sequence of complete hyperbolic metrics $\{ m_n\}$ on a closed surface, we will need to describe the behavior of the sequence on some specified subsurfaces. Namely we need to define a compactification of Teichm\"uller space for surfaces with boundary.

The fastest way to do so is to use the compactification by actions on $\R$-trees and Skora's Theorem. Consider a compact surface $F$ with boundary and let $\{m_n\}$ be a sequence of hyperbolic metrics on $F$ such that $\partial F$ is an union of $m_n$-geodesics for every $n$. Assume that $\{m_n\}$ does not contain a convergent subsequence and that there is a non-peripheral simple closed curve $c\subset F$ such that $\frac{\ell_{m_n}(\partial F)}{\ell_{m_n}(c)}\longrightarrow 0$.  We use $c$ to make sure that there is a non-peripheral closed curve whose length grows much faster than the length of $\partial F$. This ensures that each element of $\partial F$, viewed as a conjugacy class in $\pi_1(S)$, has a fixed point when acting on the $\R$-tree to which $\{m_n\}$ tends. Such a fixed point is necessary to use Culler-Morgan-Shalen's Theory and Skora's Theorem. Then, by \cite{mors1} and \cite{skora}, a subsequence of $\{m_n\}$ tends to an action of $\pi_1(F)$ on an $\R$-tree dual to a compact geodesic measured geodesic lamination $\lambda$.\\

\indent
Thus we say a sequence $\{m_n\}$ of hyperbolic metrics with geodesic boundary on $F$ {\em tends to a lamination $L$} on $F$ if there is a non-peripheral simple closed curve $c\subset F$ such that $\frac{\ell_{m_n}(\partial F)}{\ell_{m_n}(c)}\longrightarrow 0$ and any subsequence contains a further subsequence converging to a projective measured geodesic lamination supported by $L$.\\

\indent
For technical reasons, we will need to associate to a metric $m$ on $F$ a measured geodesic lamination, as was done for closed surfaces. Let $m$ be a hyperbolic metric on $F$ such that $\partial F$ is an union of closed geodesics. Let $P_F$ be a pants decomposition of $F$, so that the connected components of $F-(P_F\cup\partial F)$ are three-holed spheres. (Hence, we do not consider $\partial F$ to be contained in $P_F$.)  Fix $\eps>0$ so that $\ell_m(c)\geq\eps$ for every leaf $c$ of $P_F$. We are especially interested in sequences of metrics on $F$ with arbitrarily short boundary curves. In particular, we cannot simply repeat the construction of ${\cal F}_P(m)$ in each component of $F-P_F$ as we did in the case of a closed surface, as we won't then satisfy the hypotheses of Theorem \ref{compactlam}.

Consider the surface $DF$ obtained by doubling $F$ along its boundary, i.e. $DF$ is obtained by taking $F$ and its mirror image and by identifying the corresponding boundaries of these two surfaces. Endowing $F\subset DF$ with a hyperbolic metric $m$ with geodesic boundary and its mirror image with the mirror image of $m$, we get a complete hyperbolic metric $Dm$ on $DF$. We denote by $\partial F\subset DF$ the multi-curve corresponding to the identified boundaries of $F$ and its mirror image. We consider the pants decomposition $P_F$ and its mirror image; the union yields a multi-curve $DP_F\subset DF$.  There is a natural involution $\tau:DF\rightarrow DF$ that exchanges $F$ with its mirror image. By construction we have $\tau(DP_F)=DP_F$. We complete $DP_F$ into a pants decomposition $P_{DF}$ so that we have $\tau(P_{DF})=P_{DF}$ (see figure \ref{doublepants}). There are two possibilities for a component of $DF-DP_F$ that intersects $\partial F$, and figure \ref{doublepants} shows how to extend $DP_F$ in both cases. 

\begin{figure}[hbtp]
\psfrag{a}{$\partial F$}
\psfrag{b}{$DP_F$}
\psfrag{c}{$P_{DF}$}
\centerline{\includegraphics{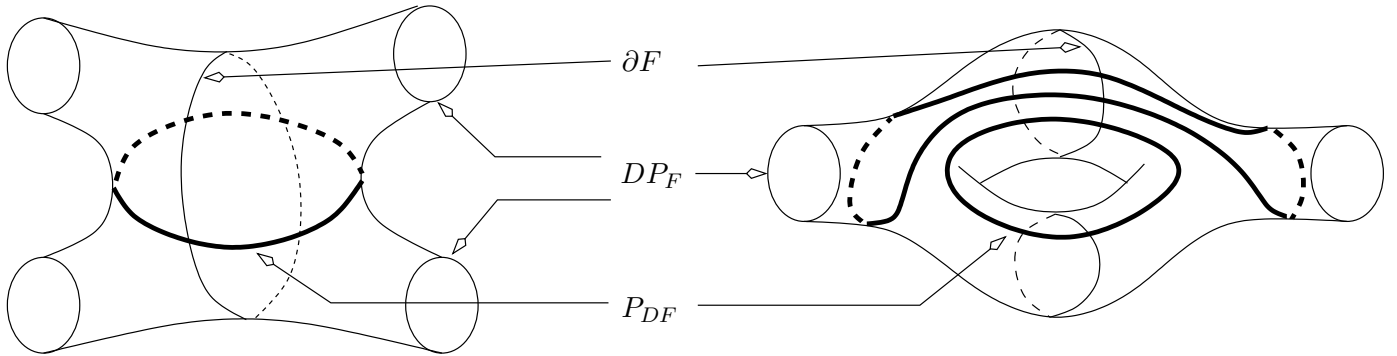}}
\caption{Extending $DP_F$ to a pants decomposition $P_{DF}$}
\label{doublepants}
\end{figure} 

Now we consider the map ${\cal F}_{P_{DF}}:{\cal T}(DF)\rightarrow {\cal ML}(DF)$ associated to $P_{DF}$ as defined in Section \ref{compactification}. When we have $i({\cal F}_{P_{DF}}(Dm),\partial F)=0$, then the restriction of ${\cal F}_{P_{DF}}(m)$ to the copy of $F$ used to construct $DF$ is a measured geodesic lamination ${\cal F}_{P_{DF}}(m)$. The metric $m$ is uniquely defined by ${\cal F}_{P_{DF}}(m)$ and the lengths of the components of $\partial F$. Thus we have a well-defined injective map from the set of metrics on $F$ satisfying $i({\cal F}_{P_{DF}}(Dm),\partial F)=0$ to ${\cal ML}({\rm int} (F))\times\R_+^k$ (where $k$ is the number of components of $\partial F$). 

Let us show that when $\ell_m(\partial F)$ is small enough, we have $i({\cal F}_{P_{DF}}(Dm),\partial F)=0$.

\begin{claim}
Given $\eps >0$, there exists $\eta>0$ depending only on $\eps$ such that if $\ell_m(\partial F)\leq\eta$ then $i({\cal F}_{P_{DF}}(Dm),\partial F)=0$.
\end{claim}

\begin{proof}
Consider the partial measured foliation ${\cal F}_{P_{DF}}(Dm)$ as originally defined and let $|{\cal F}_{P_{DF}}(Dm)|$ be its support. By \cite[Expos\'e 8]{flp}, there is a uniform bound $K$ on the length of the closure of every leaf of $|{\cal F}_{P_{DF}}(Dm)|\cap (DF-P_{DF})$, depending only on $\eps$. By the Margulis Lemma, we can find $\eta$ so that every component of $DP_F$ is at distance at least $K+1$ from $\partial F$. With this choice of $\eta$, the closure of a leaf of $|{\cal F}_{P_{DF}}(Dm)|\cap (DF-P_{DF})$ with one endpoint on $DP_F$ cannot intersect $\partial F$.

Since $P_{DF}$ is invariant under the action of $\tau$, we have $\tau({\cal F}_{P_{DF}}(Dm))={\cal F}_{P_{DF}}(Dm)$. Furthermore $\partial F$ is fixed pointwise by the action of $\tau$. It follows that the closure of a leaf of $|{\cal F}_{P_{DF}}(Dm)|\cap (DF-P_{DF})$ intersects $\partial F$ only if it has both endpoints in $DP_F$.

Combining these two paragraphs, we conclude that for $\eta$ small enough no leaf of ${\cal F}_{P_{DF}}(Dm)$ crosses $\partial F$. In particular, we have $i({\cal F}_{P_{DF}}(Dm),\partial F)=0$.
\end{proof}

Given a sequence of metrics $\{ m_n\}$ on $F$ such that $\ell_{m_n}(\partial F)\longrightarrow 0$, we choose a pants decomposition $P_F$ (not containing $\partial F$) so that we have $\ell_{m_n}(c)\geq\eps$ for every leaf $c$ of $P_F$ and every $n$. We construct a map ${\cal F}_{P_{DF}}:\{m_n\}\rightarrow {\cal ML}(F)$ as described previously. Taking a subsequence $\{ m_n\}$ such that $\{ {\cal F}_{P_{DF}}(m_n)\}$ converges projectively, we get a projective measured geodesic lamination $[\lambda]$. Since Theorem \ref{compactlam} holds for $Dm_n$ and ${\cal F}_{P_{DF}}(Dm_n)$, it is still true when applied to $m_n$ and ${\cal F}_{P_{DF}}(m_n)$. It follows that $\{m_n\}$ tends to an action on an $\R$-tree which is dual to $[\lambda]$.

A complete metric $m$ on an open surface $F$ can be approximated by a sequence of metrics $m_n$ on the surface $\overline{F}$ obtained by adding a simple closed curve along each cusp of $F$. The metrics $m_n$ have the property that $\partial\overline{F}$ is an union of $m_n$-geodesics whose lengths converge to $0$. Choosing an appropriate pants decomposition $P$ of $F$ (again not containing $\partial F$), one can define measured geodesic laminations ${\cal F}_{P_{DF}}(m_n)$ associated to $\{m_n\}$ as above. It is easy to see that $\{ {\cal F}_{P_{DF}}(m_n)\}$ converges to a measured geodesic lamination ${\cal F}_{P_{DF}}(m)$. Notice that ${\cal F}_{P_{DF}}(m)$ depends on $P$ but not on the choice of the sequence $\{m_n\}$. We have thus defined a map ${\cal F}_{P_{DF}}:{\cal T}(F)\rightarrow {\cal ML}(F)$ for which Theorem \ref{compactlam} holds.

\subsection{Gallimaufries}
\label{gall}

Let $S$ be a compact orientable surface of genus at least $2$, not necessarily connected. A {\em gallimaufry} $\Gamma=(F,m,L)$ on $S$ is made as follows: $F$ is an open incompressible subsurface of $S$, namely if a simple closed curve $d\subset F$ bounds a disc in $S$, $d$ bounds a disc in $F$; $m$ is a complete hyperbolic metric (up to isotopy) on $F$ with finite area, so that in particular a connected component of $F$ endowed with $m$ is either a surface with cusps or a connected component of $S$; and $L$ is a recurrent geodesic lamination on the compact surface $S-F$. We furthermore require that $L$ does not have any closed leaf and that the connected components of $S-(F\cup L)$ are discs or annuli. The surface $F$ is the  {\em moderate surface} of $\Gamma$ and the lamination $L$ forms its {\em immoderate lamination}. We denote by ${\cal GA}(S)$ the set of all gallimaufries on $S$.  As an example, the  end invariants of a hyperbolic $3$-manifold form a gallimaufry: the union of the open surfaces facing the geometrically finite ends forms its moderate surface, equipped with the metric induced by the corresponding conformal structure at infinity; the union of the ending laminations of its geometrically infinite ends is its immoderate lamination. We  now define the topology on ${\cal GA}(S)$.

Let $\{ \Gamma_n=(F_n,m_n,L_n)\}$ be a sequence of gallimaufries on $S$. For each component of $S-F_n$ which is an annulus $A$, we add to $L_n$ an essential simple closed curve that can be homotoped into $A$. Thus we get a new geodesic lamination $L'_n$. We  say that $\{\Gamma_n\}$ converges to $\Gamma_\infty=(F_\infty,m_\infty,L_\infty)$ in ${\cal GA}(S)$ if the following hold:
\begin{enumerate}[i)]
\item for every $n$ we have $F_\infty\subset F_n$ and the restrictions of the $m_n$ to $F_\infty$ converge to $m_\infty$, namely we have $l_{m_n}(\partial\overline{F}_\infty)\longrightarrow 0$ and for every non-peripheral closed curve $c\subset F_\infty$, we have $\ell_{m_n}(c)\longrightarrow \ell_{m_\infty}(c)$,
\item the recurrent part of the Hausdorff limit of every convergent subsequence of $\{L'_n\}$ lies in $L'_\infty$,
\item if a component $L$ of $L_\infty$ lies in infinitely many $F_n$ then the restrictions of the $m_n$ to $S(L)$ tend to $L$.   
\end{enumerate}

A gallimaufry with empty immoderate lamination is simply a point in the Teichm\"uller space of $S$. Thus by defining a gallimaufry we have constructed a bordification of Teichm\"uller space. Notice that ${\cal GA}(S)$ is not compact and not even locally compact.

\begin{lemma}           \label{locom2}
The space ${\cal GA}(S)$ is not locally compact.
\end{lemma}

\begin{proof}

Consider a simple closed curve $c\subset S$ and an embedded annulus $A\subset S$ around $c$. Let $m$ be a complete hyperbolic metric on the open surface $S-A$ and let $\{m_n\}$ be a sequence of complete hyperbolic metrics on $S$ converging to $m$ on $S-A$ in the sense of i) above. In particular since $c$ corresponds to a cusp of $S-A$, we have $l_{m_n}(c)\longrightarrow 0$.

Let $\phi: S\rightarrow S$ be a (right) Dehn twist along $c$. By construction the sequence of gallimaufries $\{(S,m_n,\emptyset)\}$ converges to $(S-A,m,\emptyset)$ and since we have $\ell_{m_n}(c)\longrightarrow 0$, for every fixed $k$, the sequence $\{(S,\phi^k_* m_n,\emptyset)\}$ also converges to $(S-A,m,\emptyset)$. On the other hand if we fix $n$ and let $k\rightarrow\infty$, no subsequence of $\{(S,\phi^k_* m_n,\emptyset)\}$ converges in ${\cal GA}(S)$ for the following reason. First notice that if a sequence of metrics $\{ m_n\}$ converges to a gallimaufry $\Gamma_\infty=(F_\infty,m_\infty,L_\infty)$, it follows from the definitions that we have $l_{m_n}(\partial F_\infty)\longrightarrow 0$. However, the $\phi^k_* m_n$-length of any given simple closed curve is bounded away from $0$, when $n$ is fixed and $k$ varies. It follows that if $\{(S,\phi^k_* m_n,\emptyset)\}$ converges to a gallimaufry $\Gamma_\infty=(F_\infty,m_\infty,L_\infty)$ then we would have $F_\infty=S$ and $L_\infty=\emptyset$. This would mean that the metrics $\{ \phi^k_* m_n\}$ converge, when $n$ is fixed and $k$ varies, which is clearly not the case.

Thus for every given compact set $K$, we can find a sequence $\{k_n\}$ so that the sequence  $\{(S,\phi^{k_n}_* m_n,\emptyset)\}$ eventually exits $K$. On the other hand, the sequence $\{(S,\phi^{k_n}_* m_n,\emptyset)\}$ converges to $(S-A,m,\emptyset)$. Hence we have shown that ${\cal GA}(S)$ is not locally compact.
\end{proof}

\subsection{Double incompressibility}
\label{double incomp}

Let $M$ be a compact $3$-manifold and let $F\subset\partial M$ be a compact surface whose boundary is incompressible. We  say that a geodesic lamination $L\subset F$ is {\em doubly incompressible in $F$} if it contains the support of a measured geodesic lamination $\lambda$ satisfying the following condition: there exists $\eta>0$ such that for every essential disc, annulus or M\"obius band $(E,\partial E)$ properly embedded in $(M,F)$, we have $i(\lambda,\partial E)\geq\eta$.

Notice that if a measured geodesic lamination $L$ is doubly incompressible, then every measured geodesic lamination with support $L$ is also doubly incompressible.

Say that a gallimaufry $\Gamma=(F,m,L)$ is {\em doubly incompressible} if $L$ is doubly incompressible in $\partial M-F$.  The  end invariants of a compact $3$-manifold form a doubly incompressible gallimaufry, see \cite{bouts} and \cite{canbouts}.

We give this definition in the most general case but, since we will only consider orientable manifolds with incompressible boundary, we will not need to consider essential discs or M\"obius bands in the definition of doubly incompressibility. Notice that although an orientable $3$-manifold may contain some essential M\"obius bands, a regular neighborhood of such a band contains an essential annulus and thus for every essential M\"obius band $E\subset M$, there is an essential annulus $A\subset M$ such that $i(\lambda,\partial A)=2i(\lambda,\partial E)$.

\section{Algebraic convergence}
\label{algebric}

In this Section we  prove that the convergence of the  end invariants to a doubly incompressible gallimaufry implies the algebraic convergence of the corresponding representations (up to taking a subsequence).

\begin{proposition}    \label{algcon}
Let $M$ be a compact. orientable, hyperbolizable $3$-manifold with incompressible boundary. Let $\{\rho_n\}\subset {\rm AH}(M)$ be a sequence of representations uniformizing $M$ and let $\Gamma_n=(F_n,m_n,L_n)$ be the  end invariants of $\rho_n (\pi_1(M))$. Assume that $\{\Gamma_n\}$ converges in ${\cal GA}(\partial M)$ to a doubly incompressible  gallimaufry $\Gamma_\infty=(F_\infty,m_\infty,L_\infty)$. Then a subsequence of $\{\rho_n\}$ converges algebraically.
\end{proposition}

This result should be regarded as another extension of Thurston's Double Limit Theorem (\cite{thuii}). The improvement from previous results (for example from \cite{ohshika-lim} which seems to be the closest one) lies in the fact that we are considering the limit of $\{\Gamma_n\}$ in the space of gallimaufries rather than its projective limit (however it may be defined). It is not hard to construct an example where $\{\Gamma_n\}$ converges to a doubly incompressible gallimaufry even though  the projective limit of the end invariants is not doubly incompressible.

We  prove this proposition by contradiction, roughly following the plan of Otal's proof of the Double Limit Theorem, see \cite{meister2}. Assume that no subsequence of $\{\rho_n\}$ converges algebraically. By work of Morgan and Shalen, see \cite{mors1}, a subsequence of $\{\rho_n\}$ tends to a small minimal action of $\pi_1(M)$ on an $\R$-tree ${\cal T}$. Consider a component $L$ of the immoderate lamination $L_\infty$ of $\Gamma_\infty$. Using the assumption that $\{\Gamma_n\}$ converges to $\Gamma_\infty$ we  construct a sequence of laminations sufficiently close to $L$ whose lengths are sufficiently well controlled. We  then deduce from the work of Otal, see \cite{conti}, that $L$ cannot be realized in ${\cal T}$. On the other hand we  see that since $\Gamma_\infty$ is doubly incompressible, at least one component of its immoderate lamination is realized in ${\cal T}$. This  yields the expected contradiction.

\subsection{Cut and paste}  
\label{cut}

Before starting the proof of Proposition \ref{algcon}, we will describe a relatively straightforward cut and paste operation that will be used many times throughout this paper.

Consider a closed orientable surface $S$ and an essential open subsurface $F\subset S$ which is not a pair of pants. Let $c\subset S$ be a  simple closed curve that intersects $F$. We use a classical construction to get a simple closed curve $e\subset F$ that behaves somewhat like $c$.  If $c$ lies in $F$ then we take $e=c$. Otherwise, let $k$ be a component of $c\cap\overline{F}$; it is an arc joining two boundary components of $\overline{F}$ (which may not be distinct). Let ${\cal V}$ be a small neighborhood of the union of $k$ and of the components of $\partial \overline{F}$ containing the endpoints of $k$. The boundary of ${\cal V}\cap F$ contains one or two simple closed curves, depending on whether the endpoints of $k$ lie in different components of $\partial \overline{F}$. Since $F$ is not a pair of pants, at least one of these curves is not peripheral. Let $e$ be such a simple closed curve, namely $e$ is freely homotopic to a component of $\partial {\cal V}$ and is not peripheral in $\overline{F}$.

By construction we have:

\begin{claim}
The simple closed curve $e$ satisfies the inequalities $\ell_{s}(e)\leq 2 \ell_{s}(c)+\ell_{s}(\partial\overline{F})$ and $i(e,c)\leq i(c,\partial\overline{F})$.\hfill $\Box$
\end{claim}

Consider now a sequence of simple closed curves $\{ c_n\}\subset S$ and the sequence of simple closed curves $\{ e_n\}\subset F$ produced by the operation above. Extract subsequences so that $\{ c_n\}$ and $\{ e_n\}$ converge in the Hausdorff topology to geodesic laminations $C$ and $E$ respectively. Assuming that a minimal sublamination $L$ of $C$ fills $F$, we have:

\begin{claim}   \label{limit}
The lamination $L$ is a sublamination of $E$.
\end{claim}

\begin{proof}
Put a weight $u_n$ on $c_n$ so that $\{u_n c_n\cap F\}$ converges to a measured geodesic lamination $\lambda$ supported by $L$. With such weights, we have $i(u_n c_n,\partial \overline{F})\longrightarrow 0$. Since we have $i(e_n,c_n)\leq i(c_n,\partial\overline{F})$  we get $i(u_n c_n,e_n)\longrightarrow 0$. Put a weight $v_n$ on $e_n$ so that $\{ v_n e_n\}$ converges to a measured geodesic lamination $\mu$. The support of $\mu$ is a sublamination of $E$. Since $\{ v_n\}$ is bounded, we have $i(\lambda,\mu)=\lim i(u_n c_n,v_n e_n)\longrightarrow 0$. Since we have assumed that the support $L$ of $\lambda$ fills $F$, then $\lambda$ and $\mu$ have the same support and $L$  is a sublamination of $E$.
\end{proof}  

Using this Claim, we will describe the behavior of a sequence of simple closed curves with bounded length in a sequence of metrics that degenerates.

\begin{lemma}    \label{shortdegen}
Let $\{m_n\}$ be a sequence of hyperbolic metrics on a closed surface $S$. Let $F\subset S$ be an incompressible subsurface such that the restrictions of the $\{m_n\}$ to $F$ converge to a geodesic lamination $L$ (in the sense of Section \ref{withboundary}) that fills $F$. Let $\{ c_n\}$ be a sequence of simple closed curves on $S$ such that $\{ \ell_{m_n}(c_n)\}$ is a bounded sequence and that $\{ c_n\}$ intersects $S(L)$ for $n$ large enough. Extract a subsequence such that $\{c_n\}$ converges in the Hausdorff topology to a geodesic lamination $C$. Then $L$ is a sublamination of $C$.  
\end{lemma}

\begin{proof}
Using the cut and paste construction described above, we get a sequence of simple closed curves $\{ e_n\subset S(L)\}$ with $\ell_{m_n}(e_n)\leq 2 \ell_{m_n}(c_n)+\ell_{m_n}(\partial\overline{S(L)})$ and $i(e_n,c_n)\leq i(c_n,\partial\overline{S(L)})$. Consider a sequence $\{ u_n\}$ converging to $0$ such that the sequence $\{u_n e_n\}\subset{\cal ML}(S(L))$ converges to a measured geodesic lamination $\lambda$. By Claim \ref{limit}, the support of $\lambda$ is a sublamination of $C$.

Consider the double $D\overline{S(L)}$ of $\overline{S(L)}$ as defined in Section \ref{withboundary} and the metric $Dm_n$ on  $D\overline{S(L)}$ induced by the restriction of $m_n$ to $\overline{S(L)}$. Choose a pants decomposition $P$ of $D\overline{S(L)}$ as in Section \ref{withboundary} such that there is $\eps>0$ for which we have that $\ell_{Dm_n}(d)\geq\eps$ for every $n$ and every leaf $d$ of $P$ and that $P$ is invariant under the natural involution $\tau$ of $D\overline{S(L)}$. We have a measured geodesic lamination ${\cal F}_P(Dm_n)$ defined as in Section \ref{compactification}.

Extract a subsequence such that  $\{ |{\cal F}_P(Dm_n)|\}$ converges in the Hausdorff topology to a geodesic lamination $L'$ and we have $\tau(L')=L'$. Let us show that any component $d$ of $\partial\overline{S(L)}$ which is a leaf of $L'$ is an isolated leaf. Otherwise $L'$ would contain leaves spiraling toward $d$. Since $\tau(L')=L'$, $L'$ would contain leaves on both sides of $d$ spiraling in the same direction towards $d$. Such behavior cannot happen in a Hausdorff limit of measured geodesic laminations. It follows that any component $d$ of $\partial\overline{S(L)}$ which is a leaf of $L'$ is eventually a leaf of $\{ |{\cal F}_P(Dm_n)|\}$.

 We remove from ${\cal F}_P(Dm_n)$ every leaf that is a component of $\partial\overline{S(L)}$. Since the restrictions of the $m_n$ to $S(L)$ tend to $L$, up to extracting a subsequence there are $v_n\longrightarrow 0$ such that $\{ v_n {\cal F}_P(Dm_n)\}$ converges to a measured geodesic lamination $D\mu$. It follows from the previous paragraph that $D\mu$ is disjoint from $\partial\overline{S(L)}$. Since the restriction of $m_n$ to $F$ tends to $L$, the support of $D\mu$ is the ``double'' of $L$. By definition, we have $v_nl_{Dm_n}(\partial\overline{S(L)})\longrightarrow i(D\mu,\partial\overline{S(L)}))=0$. We denote by $e_n$ the curve defined by $e_n$ on the copy of $\overline{S(L)}$ comprising $D\overline{S(L)}$. From Theorem \ref{compactlam}, we get $i(e_n,v_n{\cal F}_P(Dm_n))\leq v_n l_{m_n}(c_n)+2v_n l_{m_n}(\partial\overline{S(L)})\longrightarrow 0$. It follows that $i(\lambda,D\mu)=\lim u_nv_n i(e_n,{\cal F}_P(Dm_n)) =0$. Hence the support of $\lambda$ is $L$. Now from the first paragraph, we can conclude that $L$ is a sublamination of $C$.
\end{proof}

\subsection{Length and realization}
\label{length realize}

Consider a component $L$ of the immoderate lamination $L_\infty$ of $\Gamma_\infty$. In this Section, we  show that $L$ cannot be realized in ${\cal T}$. To do that, we  construct a sequence of geodesic laminations with controlled lengths which are close enough (in a sense to be made precise) to $L$. We start by roughly approximating $L$ by simple closed curves of bounded length, using Lemma \ref{shortdegen}.

\begin{claim}     \label{boundedsequence}
There is a sequence $\{ c_n\}$ of simple closed curves on $S$ such that $\{ \ell_{\rho_n}(c_n^*)\}$ is bounded and so that, up to extracting a subsequence, $\{ c_n\}$ converges in the Hausdorff topology to a geodesic lamination containing $L$.
\end{claim}

\begin{proof}
We first assume that  $L\subset F_n$ for $n$ sufficiently large, where $F_n$ is the moderate surface of $\Gamma_n$. It is a classical result of  Bers \cite{bers}  that there are pants decompositions $P_n$ of the $F_n$  so that $\{\ell_{m_n}(P_n)\}$ is a bounded sequence.  Extract a subsequence such that $\{ P_n\}$ converges in the Hausdorff topology to a geodesic lamination $P_\infty$. Since $P_n$ is a pants decomposition of $F_n$ and since $F$ is an incompressible subsurface of $F_n$, for every $n$ there is a leaf $c_n$ of $P_n$ that intersects $F$. By assumption, $\{ \ell_{m_n}(c_n)\}$ is bounded. Furthermore $\{ c_n\}$ converges in the Hausdorff topology to a sublamination $C$ of $P_\infty$. By Lemma \ref{shortdegen}, $L$ is a sublamination of $C$. By \cite{bers inequality}, we have $\ell_{\rho_n}(c_n^*)\leq 2\ell_{m_n}(c_n)$. In particular $\{ \ell_{\rho_n}(c_n^*)\}$ is bounded.

If we have $L\not\subset F_n$ then $L$ lies in the Hausdorff limit of every convergent subsequence of $\{L'_n\}$ where $L'_n$ is obtained by adding to $L_n$ a curve in each component of $\partial_{\chi<0} M-F_n$ which is an annulus. By the definition of $L'_n$, if a component $B_n$ of $L'_n$ intersecting $S(L)$ is not a closed leaf, it is an ending lamination of $\rho_n$. In this case, there are curves $c_k$ such that  $\ell_{\rho_n}(c_k^*)\leq Q$ for some $Q$ depending only on $\partial M$, such that $\{ c_k\}$ converges in the Hausdorff topology to a geodesic lamination containing $B_n$. Taking a diagonal sequence we get a sequence $\{ c_n\}$ of simple closed curves so that the sequence $\{\ell_{\rho_n}(c_n^*)\}$ is bounded and so that $\{ c_n\}$ converges in the Hausdorff topology to a geodesic lamination containing $L$. If $B_n$ is a closed curve, then  $\ell_{\rho_n}(B^*_n)=0$ and we are done by taking $c_n=B_n$. This concludes the proof of Claim \ref{boundedsequence}.
\end{proof}

Now that we have this sequence $\{c_n\}$, we  use the following proposition to deduce that no component of $L$ is realized in ${\cal T}$.

\begin{lemma}           \label{realise}
Let $M$ be a compact, orientable, hyperbolizable $3$-manifold with incompressible boundary.  Let $\{\rho_n\}\subset {\rm AH}(M)$ be a sequence tending to a small minimal action of $\pi_1(M)$ on an $\R$-tree ${\cal T}$. Let $\{c_n\}\subset\partial M$ be a sequence of simple closed curves such that $\{\ell_{\rho_n}(c_n^*)\}$ is bounded. Assume that $\{ c_n\}$ converges in the Hausdorff topology to a geodesic lamination $C_\infty$ and let $C$ be an irrational minimal sublamination of $C_\infty$. Then $C$ is not realized in ${\cal T}$.
\end{lemma}

\begin{proof}

The study of the behavior of the lengths of geodesic laminations that are realized in ${\cal T}$ has already been initiated by J.-P. Otal, see \cite{conti}. His results are stated under the assumption that $M$ is a handlebody, and he considers a sequence of geodesic laminations that converges in the Hausdorff topology. But a careful look at the proof yields the following statement.

\begin{theoremnum}[Continuity Theorem \cite{conti}]        \label{contin}
Let $M$ be a compact, orientable, hyperbolizable $3$-manifold. Let $\{\rho_n\}$ be a sequence of geometrically finite representations of $\pi_1(M)$ tending to a small minimal action of $\pi_1(M)$ on an $\R$-tree ${\cal T}$. Let $\varepsilon_n\longrightarrow 0$ be such that for all $g\in\pi_1(M)$, we have $\varepsilon_n \ell_{\rho_n}(g)\longrightarrow \ell_{\cal T}(g)$ and let $L\subset\partial_{\chi<0} M$ be a minimal geodesic lamination which is realized in ${\cal T}$. Consider a geodesic lamination $E\subset S(L)$ containing $L$. Then there exists a neighborhood ${\cal V}(E )$ of $E$ and constants $Q,n_0$ such that for every simple closed curve $c\subset {\cal V}(L )$ and for every $n\geq n_0$,
$$\varepsilon_n \ell_{\rho_n}(c^*)\geq Q  \ell_{s_0}(c).$$
\end{theoremnum}

In this statement, $s_0$ is a reference metric which is used to measure the "complexity" of the curve $c$. Any complete hyperbolic metric on $\partial M$ can be chosen and $Q$ will depend on this choice.

Theorem \ref{contin} is enough to conclude the proof of Lemma \ref{realise} when $\{c_n\}$ converges in the Hausdorff topology to $C$. In order to deal with the more general case, we use the cut and paste operation described in Section \ref{cut} in $\overline{S(C)}$ on the $c_n$. This provides us with a sequence of simple closed curves $\{e_n\}\subset S(C)$ satisfying: $\ell_{s}(e_n)\leq 2 \ell_{s}(c)+\ell_{s}(\partial\overline{F})$ and $i(e_n,c)\leq i(c,\partial\overline{F})$. Furthermore, by Lemma \ref{limit}, up to extracting a subsequence, $\{e_n\}$ converges in the Hausdorff topology to a geodesic lamination $E$ containing $C$.

Now it remains to control the length in $M_n=\Hp^3/\rho_n(\pi_1(M))$ of the sequence $\{e_n\}$ thus constructed. Let $S$ be the connected component of $\partial M$ containing $C$. Let $T\subset W(C)$ be the maximal multicurve that is disjoint from $C$, $T\subset W(C)-S(C)$. Denote by $\phi:S\rightarrow S$ the mapping class that performs one left Dehn twist along each component of $T$. For a fixed $n$, the sequence $\{ \phi^k(c_n)\}$ converges as $k\longrightarrow\infty$ to a finite geodesic lamination whose non-compact leaves spiral in $W(C)-S(C)$. Extend this lamination to a finite lamination $T_n$ whose non-compact leaves spiral in $W(C)-S(C)$ and whose complementary regions are ideal triangles.  Let $f_n:S\rightarrow M_n$ be a pleated surface (see definition in \cite[\textsection 5.1]{ceg}) homotopic to the inclusion map such that $f_n$ maps every leaf of $T_n$ to a geodesic of $M_n$. The existence of such a pleated surface follows from \cite[\textsection 5.3]{ceg}. We  denote by $\ell_{f_n}(d)$ the length of a closed geodesic $d$ of $S$ endowed with the metric induced by $f_n$. If a component $c$ of $\partial \overline{S(C)}$ corresponds to a parabolic isometry in $\rho_n(\pi_1(M))$, we consider a map $f_n:S-c\rightarrow M_n$ such that the cusps of $S-c$ are mapped to the corresponding cusps of $M_n$. Such a map $f_n$ is called a {\em noded pleated surface}, see \cite{mincom}, and we set $\ell_{f_n}(c)=0$ in this case.

Given $\eps >0$, it follows from the "efficiency of pleated surfaces", see \cite[Theorem 3.3]{thuii}, that there is a constant $Q = Q(\eps)$ such that  $\ell_{f_n}(c_n\cap R_D)\leq \ell_{\rho_n}(c_n^*)+Q i(c_n, T)$, where $R_D$ is the complement of the $\eps$-Margulis tubes around the components of $\partial \overline{S(C)}$ with short length (with respect to the metric induced by $f_n$). (Compare with \cite[p. 138]{mincom}.) It follows that there is a component $k_n$ of $c_n\cap S(C)\cap R_D$ such that $\{ \ell_{f_n}(k_n)\}$ is bounded by some constant $K>0$ depending on $Q$ and on the bound on the $\ell_{\rho_n}(c_n^*)$. Using these arcs $k_n$ in the construction described in Lemma \ref{cut}, we get a lamination $E\subset S(C)$ with $C\subset E$ and a sequence of simple closed curves $\{e_n\}\subset {\cal ML}(S(C))$ such that $\{ e_n\}$ converges to $E$ in the Hausdorff topology and such that  $\ell_{f_n}(e_n)\leq 2 \ell_{f_n}(k_n)+\ell_{f_n}(\partial\overline{S(C)})$.

By the choice of $k_n$, the sequence $\{\ell_{f_n}(k_n)\}$ is bounded by $K$, and so  $\ell_{f_n}(e_n)\leq 2K+\ell_{f_n}(\partial \overline{S(C)})$. Notice that since $f_n$ realizes $\partial\overline{S(C)}$, we have $\ell_{f_n}(\partial\overline{S(C)})=\ell_{\rho_n}(\partial\overline{S(C)}^*)$. Thus we get $\ell_{\rho_n}(e_n^*)\leq 2K+ \ell_{\rho_n}(\partial\overline{S(C)}^*)$. Since the action of $\rho_n(\pi_1(M))$ tends to the action of $\pi_1(M)$ on the $\R$-tree ${\cal T}$, there is a sequence $\eps_n\longrightarrow 0$ such that we have $\eps_n \ell_{\rho_n}(g)\longrightarrow \ell_{\cal T}(g)$ for every $g\in\pi_1(M)$.  In particular, we have $\eps_n \ell_{\rho_n}(\partial\overline{S(C)}^*)\longrightarrow \ell_{\cal T}(\partial\overline{S(C)}$. Thus $\eps_n \ell_{\rho_n}(e_n^*)\leq2\eps_nK+ \eps_n\ell_{\rho_n}(\partial\overline{S(C)}^*)$ is bounded.

On the other hand, since $\{ e_n\}$ converges to $E\supset C$, we have $\ell_{s_0}(e_n)\longrightarrow\infty$ for every complete hyperbolic metric $s_0$ on $S$. It follows then from Theorem \ref{contin} that $C\subset E$ is not realized in ${\cal T}$. This concludes the proof of Lemma \ref{realise}.
\end{proof}

Combining Claim \ref{boundedsequence} and Lemma \ref{realise}, we conclude that no component of $L_\infty$ is realized in ${\cal T}$.

\subsection{Double incompressibility and realization}
\label{double realize}

Now we show that when $\Gamma_\infty$ is doubly incompressible, at least one minimal sublamination of $L_\infty$ is realized. Thus we  get a contradiction with Claim \ref{boundedsequence} and Lemma \ref{realise}.

\begin{lemma}    \label{realisee}
Let $M$ be a compact, orientable, hyperbolizable $3$-manifold with incompressible boundary. Consider a sequence of representations $\{\rho_n\}\subset {\rm AH}(M)$ uniformizing $M$ that tends to a small minimal action of $\pi_1(M)$ on an $\R$-tree ${\cal T}$. Let  $\Gamma=(F,m,L)$ be a doubly incompressible gallimaufry. Assume that for every simple closed curve $c\subset F$, the sequence $\{\ell_{\rho_n}(c^*)\}$ is bounded. Then at least one component of $L$ is realized in ${\cal T}$.
\end{lemma}

\begin{proof}
Consider a component $S$ of $\partial_{\chi <0} M$. Since $\partial M$ is incompressible, we can view $\pi_1(S)$ as a subgroup of $\pi_1(M)$. Thus we have a small action of $\pi_1(S)$ on ${\cal T}$. Let ${\cal T}_S$ be the minimal sub-tree of ${\cal T}$ for this action. By Skora's Theorem, see \cite{skora}, this action is dual to a measured geodesic lamination $\beta$. Doing the same for each component of $\partial M$ we get a measured geodesic lamination $\beta\in{\cal ML}(\partial M)$. By \cite{mors3}, $\beta$ is not trivial. By \cite{conti}, if a minimal geodesic lamination crosses $|\beta|$, then it is realized in ${\cal T}$. So we have to show that one component of $L$ crosses $|\beta|$.

For a simple closed curve $c\subset F$, we have assumed that $\{ \ell_{\rho_n}(c^*)\} $ is bounded. Hence we have $i(c,\beta)=0$ for every simple closed curve $c\subset F$. This is possible only if $|\beta|$ lies in $\partial M-F$. Let us show that $|\beta|$ crosses $L$. 

By \cite{mors3}, $(M,S(\beta))$ is not acylindrical, i.e. there is at least one essential annulus with boundary in $S(\beta)$. In particular, if $\beta$ is a multi-curve, it contains the boundary of an essential annulus or M\"obius band. Since $L$ is doubly incompressible, it follows that $L$ crosses $\beta$.

Now we can assume that $\beta$ is not a multi-curve and denote by $\mu$ a connected sublamination of $\beta$ which is not a simple closed curve. By \cite{mors3}, $\beta$ lies in the boundary of an essential $I$-bundle $W\subset M$ and $\beta\cap W\cap \partial M$ factors through the fibration. Namely, if we denote by $B$ the base surface of $W$ and by $p:W\rightarrow B$ the projection along the fibers, then $p^{-1}(p(\mu))\cap\partial W$ is a sublamination of $\beta$ (compare with \cite{meister1}, see also  \cite[Lemme 4.7]{espoir}). Consider a sequence $e_n\subset B$ of simple closed curves that converge in the Hausdorff topology to the support of $p(\mu)$. Then $\{ E_n=p^{-1}(e_n)\}$ is a sequence of 
essential annuli or M\"obius bands such that $\{\partial E_n\}$ converges in the Hausdorff topology to a sublamination of $|\beta|$. Since $L$ is doubly incompressible in $(M,\partial M-F)$, $L$ crosses $|\beta|$.


Thus we have proved that $|\beta|$ crosses $L$. By \cite[Theorem 3.1.4]{conti} this concludes the proof of Lemma  \ref{realisee}
\end{proof}

We can now conclude the proof of Proposition  \ref{algcon}.

\begin{proof}[Proof of Proposition \ref{algcon}]
  Assume that the conclusion is not satisfied, namely that a subsequence of $\{\rho_n\}$ tends to an action of $\pi_1(M)$ on an $\R$-tree ${\cal T}$. By Claim \ref{boundedsequence} and Lemma \ref{realise} no component of $L_\infty$ is realized in ${\cal T}$. By the definition of the topology on ${\cal GA}(\partial M)$, for every simple closed curve $c\subset F_\infty$, the sequence $\{\ell_{m_n}(c)\}$ is bounded. By \cite{bers inequality}, the sequence $\{\ell_{\rho_n}(c^*)\}$ is also bounded. Thus the hypotheses of Lemma \ref{realisee} are fulfilled. It follows that at least one component of $L_\infty$ is realized in ${\cal T}$. We conclude from this contradiction that a subsequence of $\{\rho_n\}$ converges algebraically (up to conjugacy).
\end{proof}

\section{Strong convergence}
\label{strong}

Let $M$ be a compact, orientable, hyperbolizable $3$-manifold, let $\{\rho_n\}\subset {\rm SH}(M)$ be a sequence of representations uniformizing $M$ and let $\Gamma_n=(F_n,m_n,L_n)$ be the  end invariants of $\rho_n$. Assume that $\{\Gamma_n\}$ converges to a doubly incompressible gallimaufry $\Gamma_\infty=(F_\infty,m_\infty, L_\infty)$. We proved in the preceding Section that a subsequence of $\{\rho_n\}$ converges algebraically. We  now show that this subsequence converges strongly to its algebraic limit $\rho_\infty(\pi_1(M))$, by showing that the convex cores of the $\rho_n (\pi_1(M))$ converge to the convex core of $\rho_\infty (\pi_1(M))$.

\subsection{The ends of $M_\infty$}
\label{ends m inf}

 We  start with the geometrically finite ends of $\rho_\infty  (\pi_1(M))$.

\begin{lemma}           \label{geof}
Let $F$ be a connected component of the moderate surface $F_\infty$ of $\Gamma_\infty$. Then there is a convex pleated surface $f_\infty:F\rightarrow M_\infty=\Hp^3/\rho_\infty(\pi_1(M))$ homotopic to the inclusion $F\subset M$.
\end{lemma}

\begin{proof}
For $n$ sufficiently large, $F$ is a subsurface of $F_n$;  let $H_n$ be the connected component of $F_n$ containing $F$. We know that the restrictions of the $m_n$ to $F$ converge to a complete hyperbolic metric $m$ on $F$. Let $C_{\rho_n}$ be the convex core of $\rho_n(\pi_1(M))$. Since $H_n$ is a component of $F_n$, there is a convex pleated surface $f_n: H_n\rightarrow \partial C_{\rho_n}$. Furthermore, by \cite{sull}, $H_n$ endowed with the metric induced by $f_n$ is bilipschitz to $(H_n,m_n)$ with a uniform bilipschitz constant. It follows that  $\ell_{f_n}(\partial\overline{F})\longrightarrow 0$ and that $\{\ell_{f_n}(c)\} $ is bounded for every simple closed curve $c\subset F$, where $\ell_{f_n}$ is the length function of the metric induced by $f_n$. From this we deduce that there is a subsequence such that the restrictions to $F$ of the metrics induced by the $f_n$ converge to a complete hyperbolic metric.

Let $c\subset F$ be a simple closed curve such that $\rho_\infty(c)$ is a hyperbolic isometry. Since $\{ \ell_{f_n}(c)\}$ is bounded, the distance between $c^*\subset M_n$ and $f_n(c)$ is  bounded uniformly in $n$. Using Arzela-Ascoli's Theorem as in \cite{ceg} we can extract a subsequence of $\{f_n\}$ that converges to a pleated surface $f_\infty:F\rightarrow C_{\rho_\infty}$ homotopic to the inclusion $F\subset M$. Since the $f_n$ are convex surfaces, by \cite{meister1}, $f_\infty$ is a convex surface as well, see also \cite{nuite}.
\end{proof}

Next we  show that each component of the immoderate lamination $L_\infty$ of $\Gamma_\infty$ is an ending lamination of $M_\infty$.

\begin{lemma}   \label{ire}
Let $L$ be a connected component of $L_\infty$. Then there is a geometrically infinite end $E$ of $M_\infty=\Hp^3/{\rho_\infty}(\pi_1(M))$ such that $E$ is homeomorphic to $S(L)\times [0,\infty)$, the inclusion $E\rightarrow M_\infty$ is homotopic to the inclusion $S(L)\hookrightarrow M$ and $L$ is the ending lamination of $E$.
\end{lemma}
\begin{proof}
Let $S$ be the component of $\partial M$ containing $L$ and let $\sigma_n:\pi_1(S)\rightarrow {\rm PSL}(2,\C)$ be the representation induced from $\rho_n$ by the inclusion map. We note that  $\ell_{\sigma_n}(\mu^*)=\ell_{\rho_n}(\mu^*)$ for every measured geodesic lamination $\mu\in{\cal ML}(S)\subset{\cal ML}(\partial M)$.

By Claim \ref{boundedsequence}, there is a sequence $\{c_n\}$ of simple closed curves converging in the Hausdorff topology to a geodesic lamination containing $L$ such that $\{\ell_{\sigma_n}(c^*_n)\}$ is a bounded sequence. Choose a transverse measure $\lambda$ supported by $L$. Since $\{\ell_{\sigma_n}(c^*_n)\}$ is a bounded sequence, it follows from the continuity of the length function, see \cite{brock}, that  $\ell_{\sigma_\infty}(\lambda^*)=0$. This means that $L$ is not realized in $\Hp^3/\sigma_{\infty}(\pi_1(S))$. It follows that $L$ is an ending lamination of an end $E'$ of $\Hp^3/\sigma_\infty(\pi_1(S))$

By \cite{canbouts}, the end $E'$ covers an end $E$ of $M_\infty$ and the covering $E'\rightarrow E$ is finite-to-one. On the other hand, $E'$ is homeomorphic to $S(L)\times [0,\infty)$ (this comes from the fact that the ending lamination has to "fill up" the surface defining the end) and if we consider the surface $S(L)\times \{1\}$ its image in $E$ under the covering $E'\rightarrow E$ is homotopic to the inclusion $S(L)\hookrightarrow M$. Therefore the covering $E'\rightarrow E$ is a homeomorphism. Thus we have proved that there is a geometrically infinite end $E$ of $M_\infty=\Hp^3/{\rho_\infty}(\pi_1(M))$ such that $E$ is homeomorphic to $S(L)\times [0,\infty)$, the inclusion $E\hookrightarrow M_\infty$ is homotopic to the inclusion $S(L)\hookrightarrow M$ and $L$ is the ending lamination of $E$.

\end{proof}

\subsection{Reconstructing the convex core}
\label{reconstructing}

In this Section we  show how the results of the preceding Section allow us to describe the convex core of $M_\infty$.

Let $L$ be a component of $L_\infty$. By Lemma \ref{ire}, $M_\infty$ has a geometrically infinite end homeomorphic to $S(L)\times[0,\infty)$ with ending lamination $L$. Choose some $p >0$ and consider the map $f_\infty: S(L)\rightarrow S(L)\times\{p\}$. Each cusp of $S(L)$ is mapped under $f_\infty$ to a cusp of $M_\infty$. Let $G_\infty$ be the union of the surfaces $S(L)$ when $L$ runs through all the components of $L_\infty$. We have thus constructed an embedding $f_\infty:G_\infty\rightarrow M_\infty$ which is homotopic to the inclusion map.

In Lemma \ref{geof},  we defined a map $f_\infty:F_\infty\rightarrow C_{\rho_\infty}$ which is a homeomorphism onto its image and is homotopic to the inclusion map. Now we have a map $f_\infty:F_\infty\cup G_\infty\rightarrow M_\infty$. The complementary regions of $F_\infty\cup G_\infty$ in $\partial_{\chi<0} M$ are annuli. By Lemmas \ref{geof} and \ref{ire}, the simple closed curve in the homotopy class defined by each of these annuli corresponds to a (maximal) parabolic conjugacy class of $\rho_\infty(\pi_1(M))$. Furthermore, since $\Gamma_\infty$ is doubly incompressible, to each such parabolic element of $\rho_\infty(\pi_1(M))$ there corresponds exactly one component of $\partial M-(F_\infty\cup G_\infty)$. To each such component (which is an annulus) corresponds two cusps of $F_\infty\cup G_\infty$ whose images under $f_\infty$ are two homotopic non-compact annuli lying in a cusp of $M_\infty$. Remove from $f_\infty(F_\infty\cup G_\infty)$ these two non-compact annuli and connect the boundary components so created by a compact annulus. Perform the same operation for all the components of $\partial_{\chi<0} M-(F_\infty\cup G_\infty)$. We get a compact surface $S_\infty\subset M_\infty$.

Change $f_\infty$ to get a homeomorphism $g_\infty:\partial_{\chi<0} M\rightarrow S_\infty$ (this only involves making the correct choice of the Dehn twisting in $\partial_{\chi<0} M-(F_\infty\cup G_\infty)$ so that $g_\infty$ is homotopic to the inclusion $\partial_{\chi<0} M\hookrightarrow M_\infty$). Adjoin to $S_\infty$ the boundary of the rank $2$ cusps of $M_\infty$ and extend $g_\infty$ to $\partial M$. Now we have a homeomorphism $g_\infty:\partial M\rightarrow S_\infty$ which is homotopic to the identity. Since $\partial M$ bounds a compact $3$-manifold, so does $g_\infty(\partial M)$. We deduce easily from this that $f_\infty (F_\infty\cup G_\infty)$ bounds a subset $C_\infty$ of $M_\infty$ which has a finite volume (compare with \cite[Lemme 21]{meister1}). Since $f_\infty(S(L_\infty))$ bounds an union of geometrically infinite ends and since $f_\infty(F_\infty)$ is an union of convex pleated surfaces, the union of $C_\infty$ and of the  geometrically infinite ends contains the convex core $C_{\rho_\infty}$ of $\rho_\infty$. If one component of $\partial C_{\rho_\infty}$ were to lie in ${\rm int}(C_\infty)$, the corresponding geometrically finite end would lie inside $C_\infty$. This would contradict the fact that $C_\infty$ has finite volume (see \cite[Lemme 21]{meister1} for more details). It follows that the union of $C_\infty$ and of the geometrically infinite ends of $M_\infty$ is the convex core of $\rho_\infty$.

Extend $g_\infty$ to a homotopy equivalence $h$ from $M$ to the compact set bounded by $g_\infty(\partial M)$. By \cite{bois}, $h$ is homotopic to a homeomorphism. It follows that $\rho_\infty$ uniformizes $M$.

\subsection{Domain of discontinuity and strong convergence}
\label{domain strong}

Now that we know the geometrically finite ends and the behavior of the $m_n$, we can deduce from \cite{gero} that the algebraically convergent sequence $\{\rho_n\}$ converges strongly. On the other hand, as we have already proved the convergence of the convex cores, it is not hard now to conclude directly that we have strong convergence. So we briefly describe here why the sequence $\{\rho_n\}$ converges strongly.

The main tool is a result of \cite{jorgmar} (see also \cite{jap}):
\begin{theoremnum}
Let $\{\rho_n\}\subset {\rm AH}(M)$ be a sequence of representations that converges algebraically to $\rho_\infty$. Assume that $\Omega_{\rho_\infty}$ is not empty. If $\{\Omega_{\rho_n} \}$ converges to $\Omega_{\rho_\infty}$ in the sense of Carath\'eodory, then $\{\rho_n\}$ converges strongly to $\rho_\infty$.
\end{theoremnum}

Recall that $\{\Omega_{\rho_n}\}$ converges to $\Omega_{\rho_\infty}$ in the sense of Carath\'eodory if and only if $\{ \Omega_{\rho_n}\}$ satisfies the two following conditions:
\begin{enumerate}[- ]
\item every compact subset $K\subset\Omega_{\rho_\infty}$ lies in $\Omega_{\rho_n}$ for all sufficiently large $n$;
\item every open set $O$ that lies in $\Omega_{\rho_n}$ for infinitely many $n$ also lies in $\Omega_{\rho_\infty}$.
\end{enumerate}

We show that the $\Omega_{\rho_n}$ and $\Omega_{\rho_\infty}$ satisfy these two conditions.

\begin{lemma}   \label{gege}
Under the hypotheses of Proposition \ref{algcon}, a subsequence of $\{\rho_n\}$ converges strongly to $\rho_\infty$.
\end{lemma}

\begin{proof}
When $\Omega_{\rho_\infty}$ is empty, $\{\rho_n\}$ converges strongly to $\rho_\infty$ by \cite{canbouts} (see \cite{gero}).  Assume now that $\Omega_{\rho_\infty}$ is not empty.

Consider the convex pleated surface $f_n:F_n\rightarrow C_{\rho_n}$. Let $\Pi$ be a hyperbolic plane that intersects $f_\infty(F_\infty)$ in a non-degenerate subsurface. The ideal boundary of $\Pi$ in $\widehat\C = \partial\overline{\Hp}^3$ is a circle which bounds a disc $D\subset\partial\overline{\Hp}^3$ such that  ${\rm int}(D)\subset \Omega_{\rho_\infty}$.

Let $K\subset\Omega_{\rho_\infty}$ be a compact connected subset. Such a compact set is covered by the interiors of finitely many discs $D_i$ defined as above. Since $\{ f_n\}$ converges to $f_\infty$, each such disc $D_i$ is the limit of a sequence $\{D_{i,n}\}$ where $D_{i,n}\subset \Omega_{\rho_n}$ is the disc bounded by the ideal boundary of a support plane for $f_n(F_n)$. It follows that for $n$ sufficiently large, $K$ is covered by the $D_{i,n}$. In particular, we have $K\subset\Omega_{\rho_n}$ for $n$ sufficiently large.

Now, we will prove, by contradiction, that every open set $O$ that lies in $\Omega_{\rho_n}$ for infinitely many $n$ also lies in $\Omega_{\rho_\infty}$.
Let $O$ be an open set lying in  $\Omega_{\rho_n}$ for infinitely many $n$ and let $w\in O\cap \Lambda_{\rho_{\infty}}$. Since the fixed points of hyperbolic isometries are dense in $\Lambda_{\rho_{\infty}}$ and since $\{\rho_n\}$ converges to $\rho_\infty$, there is a sequence $\{w_n\}$ of points converging to $w$ such that $w_n\in \Lambda_{\rho_n}$. For $n$ sufficiently large, we have $w_n\in O\cap \Lambda_{\rho_n}$, contradicting our hypothesis that $O\subset\Omega_{\rho_n}$ for infinitely many $n$.
\end{proof}

Finally, we have now proved that the ending map is proper, namely:

\begin{proposition}    \label{strocon}
Let $M$ be a compact, orientable, hyperbolizable $3$-manifold with incompressible boundary.  Let $\{\rho_n\}\subset {\rm AH}(M)$ be a sequence of representations uniformizing $M$ and let $\Gamma_n=(F_n,m_n,L_n)$ be the  end invariants of $\rho_n$. Assume that $\{\Gamma_n\}$ converges in ${\cal GA}(\partial M)$ to a doubly incompressible  gallimaufry $\Gamma_\infty=(F_\infty,m_\infty,L_\infty)$. Then a subsequence of $\{\rho_n\}$ converges strongly.  \hfill $\Box$
\end{proposition}

\section{Necessary conditions}
\label{continuity}

\label{the converse}
In this Section we  prove that the end invariants of a sequence $\{\rho_n\}$ converge to the end invariants of the limit.

\begin{proposition}     \label{concon}
Let $M$ be a compact, orientable, hyperbolizable $3$-manifold with incompressible boundary.  Let $\{ \rho_n\}$ be a sequence of representations that uniformize $M$. Assume that $\rho_n$ is geometrically finite and minimally parabolic for all $n$ and let $\Gamma_n=(\partial_{\chi<0} M,m_n,\emptyset)$ be its end invariants. If $\{\rho_n\}$ converges strongly to a representation $\rho_\infty$, then $\{\Gamma_n\}$ converges in ${\cal GA}(\partial M)$ to the end invariants $\Gamma_\infty$ of $\rho_\infty$.
\end{proposition}

\begin{proof}
First recall that we have seen in Lemma \ref{topo} that $\rho_\infty$ uniformizes $M$. Let us also recall the topology of ${\cal GA}(\partial M)$ as described in Section \ref{gall}. Consider our sequence of gallimaufries $\{\Gamma_n=(\partial_{\chi<0} M,m_n,\emptyset)\}$. Then $\{\Gamma_n\}$ converges to $\Gamma_\infty=(F_\infty,m_\infty,L_\infty)$ in ${\cal GA}(\partial M)$ if we have the following:
\begin{enumerate}[]
\item i) for every $n$ we have $F_\infty \subset F_n$ and the restrictions of the $m_n$ to $F_\infty$ converge to $m_\infty$, namely for every closed curve $c\subset F_\infty$, we have $\ell_{m_n}(c)\longrightarrow \ell_{m_\infty}(c)$,
\item iii) if a component $L$ of $L_\infty$ lies in infinitely many of the $F_n$, then the restrictions of the $m_n$ to $S(L)$ tend to $L$.
\end{enumerate}

Notice that since the $\Gamma_n$ lie in ${\rm int}({\cal GA}(\partial M))$ the definitions are simpler than in the general case.  In particular, part $(ii)$ of the definition is trivially satisfied.

The first property can be deduced from the convergence of the limit sets.

\begin{lemma}   \label{cond1}
Consider the moderate surface $F_\infty$ of $\Gamma_\infty$ equipped with its complete hyperbolic metric $m_\infty$. For every closed curve $c\subset F_\infty$, we have $\ell_{m_n}(c)\longrightarrow \ell_{m_\infty}(c)$.
\end{lemma}

\begin{proof}
Let $F$ be the component of $F_\infty$ containing $c$ and let $S$ be the component of $\partial M$ containing $F$. Since $\{ \rho_n\}$ converges strongly, as proved in \cite{kt} and \cite{ohshika-strong}, the limit sets $\{ \Lambda_{\rho_n}\}$ of the Kleinian groups $\rho_n(\pi_1(S))$ converge to the limit set of $\rho_\infty(\pi_1(S))$ in the Hausdorff topology. It follows that $\{\Omega_{\rho_n}\}$ converges to $\Omega_{\rho_\infty}$ in the sense of Carath\'eodory, and hence that the Poincar\'e metrics on the $\Omega_{\rho_n}$ converge to the Poincar\'e metric on $\Omega_{\rho_\infty}$.   By this, we mean that if we consider a point $x\in\Omega_{\rho_\infty}$, and a sequence $\{x_n\}\subset\Omega_{\rho_n}$ converging to $x$ (whose existence is guaranteed by the convergence of the $\Lambda_{\rho_n}$) the Poincar\'e metric on $\Omega_{\rho_n}$ at $x_n$ converges to the Poincar\'e metric on $\Omega_{\rho_\infty}$ at $x$, see \cite{hejhal}. 

Consider a component $O$ of $\Omega_{\rho_\infty}$ covering $F$ and an arc $\widetilde{c}$ that is a lift of $c$ to $O$.  Let $G_O\subset\pi_1(M)$ be the subgroup such that $\rho_\infty(G_O)$ is the stabilizer of $O$.  If we let $x$ be one endpoint of $\widetilde{c}$, then the other endpoint is the point $\rho_\infty(g)(x)$, where $\rho_\infty(g)$ is the element of $\rho_\infty(G_O)$ (up to conjugacy) corresponding to $c$.  Let $K$ be the closure of the $\varepsilon$-neighborhood of $\widetilde{c}$ in the Poincar\'e metric on $O$, and note that $K$ is a compact subset of $O$.  In particular, $K\subset \Omega_{\rho_n}$ for all $n$ sufficiently large.  Since the Poincar\'e metrics on the $\Omega_{\rho_n}$ converge to the Poincar\'e metric on $\Omega_{\rho_\infty}$ uniformly on compact subsets, since $\rho_n(g)(x)$ lies in $K$ for $n$ sufficiently large and since $\{\rho_n(g)(x)\}$ converges to $\rho_\infty(g)(x)$, we see that $\ell_{m_n}(c)\longrightarrow \ell_{m_\infty}(c)$, as desired.
\end{proof}

Next we prove that for a connected component $L$ of $L_\infty$, up to extracting a subsequence, the restrictions of the $m_n$ to $S(L)$ tend to a measured geodesic lamination supported by $L$.  To prove this property, we  use some of the ideas of the "lemme d'intersection" \cite[Proposition 3.4]{bouts}. Consider a component $L$ of the immoderate lamination of $\Gamma_\infty$. By definition of the end invariants of $\rho_\infty$, $L$ is the ending lamination of a geometrically infinite end of $M_\infty$. We want to show that, up to extracting a subsequence, the restrictions of the $m_n$ to $S(L)$ tend to a measured geodesic lamination supported by $L$.

\begin{lemma}           \label{cond2}
Consider a connected component $L$ of $L_\infty$. Then, up to extracting a subsequence, the restrictions of the $m_n$ to $S(L)$ tend to a measured geodesic lamination supported by $L$. 
\end{lemma}

Before starting the proof of Lemma \ref{cond2} recall that it finishes the proof of Proposition \ref{concon}.
Putting Lemmas \ref{cond1} and \ref{cond2} together yields that given a strongly convergent sequence $\{\rho_n\}$ of representations that uniformize $M$ with  end invariants $\Gamma_n=(\partial_{\chi<0} M,m_n,\emptyset)$, we have that $\{\Gamma_n\}$ converges in ${\cal GA}(\partial M)$ to the end invariants of the limit $\rho_\infty$ . This concludes the proof of Proposition \ref{concon}.
\end{proof}

\begin{proof}[Proof of Lemma \ref{cond2}]
Let $S$ be the connected component of $\partial M$ that contains $L$. Since $L$ is a minimal component of the immoderate part of $L_\infty$ it is the ending lamination of a geometrically infinite end of $M_\infty$. Then there is a sequence $\{e_k\}\subset S(L)$ of simple closed curves  such that $\{e_k\}$ converges in ${\cal PML}(\partial M)$ to a projective measured geodesic lamination supported by $L$, the geodesic representatives $e^*_{k,\infty}$ of $e_k$ in $M_\infty$ exit every compact subset of $M_\infty$, and $\{ \ell_{\rho_\infty}(e^*_{k,\infty}) \}$ is bounded. This follows directly from \cite[Lemma 7.9]{elc1} but one can also write a simpler proof using a sequence of pleated surfaces exiting the end  faced by $S(L)$ and Bonahon's Intersection Lemma \cite[Proposition 3.4]{bouts}.

Since $\{\rho_n\}$ converges strongly to $\rho_\infty$, there are sequences $q_n\longrightarrow 1$ and $R_n\longrightarrow\infty$, a point $x_\infty\in M_\infty$, and a sequence of diffeomorphisms $\psi_n:B(x_\infty,R_n)\subset M_\infty\rightarrow M_n$ such that $\psi_n(x_\infty)=x_n$ and that $\psi_n$ is a $q_n$-bilipschitz diffeomorphism onto its image. Choose $\eps$ such that for every $k$ either $e^*_{k,\infty}$ is disjoint from the $\eps$-thin part of $M_\infty$ or $e^*_{k,\infty}$ is the core of an $\eps$-Margulis tube.

Consider a sequence $\{P_n\}$ of pants decompositions  of $S$ such that $\{\ell_{m_n}(P_n)\}$ is a bounded sequence. We show that, up to extracting a subsequence, $\{P_n\}$ converges in the Hausdorff topology to a geodesic lamination containing $L$. Then we show that this can happen only when the $m_n$ degenerate transversely to $L$. First we control the intersection numbers between $P_n$ and $e_k$ using the following Lemma.

\begin{lemma}   \label{ineq}
Consider a sequence $\{d_n\}\subset S$ of simple closed curves. There is $Q>0$ such that for every $k$ sufficiently large there is $N_k$ such that  $i(e_k,d_n)\leq Q(\ell_{m_n}(d_n)+2\pi)$ for $n\geq N_k$.
\end{lemma}

\begin{proof}
The basis of this proof comes from \cite[Proposition 3.4]{bouts}.

Let $K_\infty\subset C_{\rho_\infty}\subset M_\infty=\Hp^3/\rho_\infty(\pi_1(M))$ be a compact core for $M_\infty$. Since $\{\rho_n\}$ converges strongly to $\rho_\infty$, for $n$ sufficiently large $\psi_n(K_\infty)$ is a compact core for $M_n$ (see Lemma \ref{topo}). By \cite{kt} or \cite{ohshika-strong}, $\{ \Lambda_{\rho_n}\}$ converges to $\Lambda_{\rho_\infty}$ in the Hausdorff topology. By \cite{brian}, the convex hulls $\{ H_{\rho_n}\}\subset \Hp^3$ converge to $H_{\rho_\infty}\subset\Hp^3$ in the Hausdorff topology. It follows that we have $\psi_n(K_\infty)\subset C_{\rho_n}$ for $n$ large enough. We use the notation $K_n=\psi_n(K_\infty)$. Notice that the induced metrics on the $\partial K_n$ converge to the induced metric on $\partial K_\infty$.

Recall that $e^*_{k,n}$ is the geodesic representative of $e_k$ in $M_n$. Since $e^*_{k,\infty}$ exits every compact set, ${\rm d}(e^*_{k,\infty},K_\infty)\geq 3C_k$ for some $C_k\longrightarrow\infty$. Fix $k$. For $n$ sufficiently large, $e^*_{k,\infty}$ lies in $B(x_\infty,R_n)\subset M_\infty$. Furthermore $\ell_{\rho_n}(\psi_n(e^*_{k,\infty}))\leq q_n \ell_{\rho_\infty}(e^*_{k,\infty})$ where $q_n$ is the bilipschitz constant of $\psi_n$. So  $\ell (e^*_{k,n})\leq q_n \ell_{\rho_\infty}(e^*_{k,\infty})$ for $n$ sufficiently large. Since the length of $e^*_{k,\infty}$ is bounded, there are $Q_1$ and $N_k$ such that for $n\geq N_k$, we have $\ell_{\rho_n}(e^*_{k,n})\leq Q_1$ (note that $Q_1$ depends on the sequence $\{ e_k\}$ but not on $k$). For $n$ sufficiently large, ${\rm d}(\psi_n(e^*_{k,\infty}),K_n)\geq 2C_k$. Since $\psi_n$ is $q_n$-bilipschitz, $\psi_n(e^*_{k,\infty})$ is a $q_n^2$-quasi-geodesic. It follows that ${\rm d}(e^*_{k,n},K_n)\geq C_k$ for $n$ sufficiently large.

In what follows, we view $d_n$ and $e_k$ as simple closed curves in $S\subset \partial K_\infty$. Let $f_n:S \rightarrow\partial C_{\rho_n}$ be the map sending $S\subset\partial M$ to the corresponding part of the boundary of the convex core of $M_n=\Hp^3/\rho_n(\pi_1(M))$. Consider an annulus $A_{d_n}$ joining $\psi_n(d_n)$ to $f_n(d_n)$ in $C_{\rho_n}$ and an annulus $A_{e_{k,n}}$ joining $\psi_n(e_k)$ to $e^*_{k,n}$ in ${\rm int}(C_{\rho_n})$. Since $K_n$ is a compact core for $M_n$ and $\partial M$ is incompressible, $M_n-K_n$ is homeomorphic to $\partial M\times\R$ (see \cite{mms}). It follows that we can choose $A_{d_n}$ and $A_{e_{k,n}}$ so that they intersect $K_n$ only along $\psi_n(d_n)$ and $\psi_n(e_k)$ respectively. By \cite[Lemme 3.2]{bouts}, we have $i(e_k,d_n)\leq i(f_n(d_n), A_{e_{k,n}})+i(e^*_{k,n}, A_{d_n})$ where the first term is the geometric intersection number in $S$ and the two others are geometric intersection numbers in $M_n$ (as defined in \cite[\textsection III]{bouts}). Since $e^*_{k,n}\subset{\rm int}(C_{\rho_n})$, we can choose $A_{e_{k,n}}$ so that it lies in ${\rm int}(C_{\rho_n})$. In particular $A_{e_{k,n}}$ does not intersect $\partial C_{\rho_n}$. It follows then that  $i(f_n(d_n), A_{e_{k,n}})=0$. Hence we only need to bound $i(e^+_{k,n},A_{d_n})$.

Notice that $i(e^*_{k,n},A_{d_n})$ is invariant by homotopy (of $e^*_{k,n}$ or $A_{d_n}$) as long as $e^*_{k,n}$ does not intersect $\partial A_{d_n}$ through the homotopy (\cite[Lemma 3.1]{bouts}). We will now modify $A_{d_n}$ to make computation easier. For each $n$, we choose a point $t_n\in A_{d_n}$ such that $\{{\rm d}(K_n,t_n)\}$ is a bounded sequence. We also choose points $y_n\in\psi_n(d_n)$ and $z_n\in f_n(d_n)$. Start with the annulus $A_{d_n}\subset C_{\rho_n}$ joining $\psi_n(d_n)$ to $f_n(d_n)$ that was used in the previous paragraph. Let us change $A_{d_n}$ by a homotopy so that it contains $t_n$ and a geodesic segment $k_n$ joining $y_n$ to $z_n$. Lift this annulus to an infinite band $\tilde A_{d_n}$ in the universal cover $\Hp^3$ of $M_n$. The annulus $A_{d_n}$ is the quotient of $\tilde A_{d_n}$ under a covering transformation $\rho_n(a_n)$. The part of $\tilde A_{d_n}$ between two lifts $\tilde k_n$ and $\tilde k'_n=\rho_n(a_n)(\tilde k_n)$ of $k_n$ is a disc $\tilde D_n$ containing a lift $\tilde t_n$ of $t_n$. We change $\tilde D_n$ by a homotopy to the geodesic cone from $\tilde t_n$ to $\partial\tilde D_n$. Then we replace $\tilde A_{d_n}$ by $\bigcup_{j\in\Z}\rho_n(a_n^j)(\tilde D_n)$ and $A_{d_n}$ by the quotient of this new band. The new annulus $A_{d_n}$ is the union of the set $E_n$ made up of geodesic segments joining $\psi_n(d_n)$ to $t_n$, of the set $F_n$ made up of geodesic segments joining $k_n$ to $t_n$ and of the set $G_n$ of geodesic segments joining $f_n(d_n)$ to $t_n$. Notice that $\partial A_{d_n}$ has not been moved during these homotopies, hence $i(e^*_{k,n},A_{d_n})$ has not changed. According to the previous paragraph, we have $i(e_k,d_n)\leq i(e^*_{k,n}, A_{d_n})$.

By construction, for $n$ sufficiently large, $E_n\subset \psi_n(B(x_\infty,r_n))$ and $\psi_n^{-1}(E_n)$ lies in a compact set (independent of $n$). It follows that for $k$ sufficiently large, $e^*_{k,n}$ does not intersect $E_n$. Thus for $k$ sufficiently large, $e^*_{k,n}$ only intersects $F_n\cup G_n$. Since $d(\psi_n^{-1}(E_n),e^*_{k,\infty})\longrightarrow \infty$ as $k\longrightarrow\infty$, we even have some $\eps>0$ such that for $n$ sufficiently large the intersections between $e_{k,n}^*$ and $F_n\cup G_n$ are at distance at least $\eps$ from the two geodesic segments in $F_n\cup G_n$ joining $t_n$ to $y_n$. Furthermore, since $\{ H_{\rho_n}\}$ converges to $H_{\rho_\infty}$ in the Hausdorff topology, we have $d(e^*_{k,n},\partial C_{\rho_n})\geq\eps$, for $n$ large enough. Thus the intersections between $e_{k,n}^*$ and $F_n\cup G_n$ are at distance at least $\eps$ from $\partial (F_n\cup G_n)$.

By construction, $F_n$ is the union of two geodesic triangles. Thus we have ${\rm area}(F_n)\leq 2\pi$. Since $f_n(d_n)$ is piecewise geodesic, $G_n$ is a union of geodesic triangles. It is well known that the area of a hyperbolic triangle is less than the length of each of its edges, see \cite[Lemma 9.3.2]{notes}. Thus we have  ${\rm area}(G_n)\leq \ell (f_n(d_n))$.

From the proof of \cite[Proposition 3.4]{bouts}, we get:

\begin{enumerate}[-]
\item If $\ell (e^*_{k,n})>\eps$, then $i(e_k,d_n)\leq Q_2^{-1} \ell_{\rho_n} ( e^*_{k,n}){\rm area}(F_n\cup G_n)$, where $Q_2$ is the volume of the ball with radius $\frac{\eps}{8}$ in $\R\times\Hp^2$;
\item If $\ell (e^*_{k,n})\leq\eps$, then we have $i(e_k,d_n)\leq Q_3^{-1}{\rm area}(F_n\cup G_n)$, where $Q_3$ is the area of a hyperbolic disc with radius $\frac{\eps}{8}$.
\end{enumerate}

We have seen above the inequality 
$${\rm area}(F_n\cup G_n)\leq \ell (f_n(d_n))+2\pi$$
By \cite{ahl}, we have $\ell (f_n(d_n))\leq 2\ell_{m_n}(d_n)$. Consider $Q=\max\{2Q_1Q_2^{-1},2 Q_3^{-1}\}$ and notice that $Q$ depends on the choice of the sequence $e_k$. For $k$ sufficiently large and for $n$ sufficiently large (depending on $k$), we have  $i(e_k,d_n)\leq Q (\ell_{m_n}(d_n)+2\pi)$.
\end{proof}

We can now show that every sequence of short pants decompositions, with respect to the $m_n$, converges to $L$. More precisely, we have:

\begin{claim}           \label{shortpants}
Let $\{P_n\}$ be a sequence of pants decompositions of $S$ such that $\{\ell_{m_n}(P_n)\}$ is a bounded sequence. Extract a subsequence such that $\{ P_n\}$ converges in the Hausdorff topology to a geodesic lamination $P_\infty\subset S$. Then $L$ is a sublamination of $P_\infty$.
\end{claim}

\begin{proof}
Since $L$ is an irrational geodesic lamination, if $P_\infty$ crosses $L$, then we have $i(e_k,P_n)\longrightarrow\infty$. By Lemma \ref{ineq}, $\{ i(e_k,P_n)\}$ is bounded. It follows that either $L$ is disjoint from $P_\infty$ or $L$ is a sublamination of $P_\infty$. Since $P_n$ is a pants decomposition for every $n$, no geodesic lamination is disjoint from $P_\infty$. So we conclude that $L$ is a sublamination of $P_\infty$.
\end{proof}

Now we can conclude the proof of Lemma \ref{cond2}. As we have seen, there is a constant $K$ depending only on $S$ such that for every complete hyperbolic metric $m$ on $S$, there is a pants decomposition $P_S$ with $\ell_m(P_S)\leq K$ (see \cite{bers}). Consider a sequence $\{P_n\}$ of pants decompositions such that $\{ \ell_{m_n}(P_n)\}$ is bounded. Extract a subsequence such that $\{ P_n\}$ converges in the Hausdorff topology to a geodesic lamination $P_\infty$. By Claim \ref{shortpants}, $L$ is a sublamination of $P_\infty$. Let $c\subset S(L)$ be a non-peripheral simple closed curve. Since $c$ crosses $L$, we have $i(P_n,c)\longrightarrow\infty$. It follows then from the Collar Lemma that  $\ell_{m_n}(c)\longrightarrow\infty$. In particular the restrictions of the $m_n$ to $S(L)$ are unbounded. By Lemma \ref{shortdegen}, the restrictions of the $m_n$ to $S(L)$ tend to $L$. This concludes the proof of Lemma \ref{cond2} and of Proposition \ref{concon}.
\end{proof}

Using the Ending Lamination Classification and Propositions \ref{strocon} and \ref{concon}, we can now prove Theorem \ref{conf}.

\setcounter{theorem}{3}

  \begin{theorem} 
 Let $M$ be a compact, orientable, hyperbolizable $3$-manifold with incompressible boundary.
 The ending map that to a representation uniformizing $M$ associates its  end invariants is a homeomorphism from ${\rm SH}(M)$ into the set of doubly incompressible gallimaufries.
 \end{theorem}
 
 \begin{proof}
 By the Ending Lamination Classification, see \cite{elc1} and \cite{elc2}, the ending map is injective. By Proposition \ref{strocon}, this map is proper. Combining Propositions \ref{strocon} and \ref{concon}, we get that a sequence $\{\rho_n\}\subset {\rm SH}(M)$ with end invariants $\{\Gamma_n=(\partial_{\chi<0} M, m_n,\emptyset)\}$ converges strongly to $\rho_\infty$ with end invariants $\Gamma_\infty=(F_\infty, m_\infty, L_\infty)$ if and only if $\{\Gamma_n\}$ converges to $\Gamma_\infty$ in ${\cal GA}(\partial M)$. In particular it follows that every $\rho_\infty\in {\rm SH}(M)$ can be approximated by geometrically finite minimally parabolic representations in the strong topology, namely ${\rm SH}(M)$ is the closure of its interior. By using such approximations we can drop the assumption in Proposition \ref{concon} that $\rho_n$ is geometrically finite and minimally parabolic. It follows that the ending map is continuous. Thus the ending map is injective, continuous and proper: it is a homeomorphism onto its image.
 \end{proof}

From this Theorem and from Lemma \ref{locom2}, we deduce that ${\rm SH}(\pi_1(M))$ is not locally compact (Lemma \ref{locom}).

\setcounter{section}{1}
\setcounter{proposition}{0}

\begin{lemma}          
Let $M$ be a compact, orientable, hyperbolizable $3$-manifold with incompressible boundary. Then the space ${\rm SH}(M)$ is not locally compact.\hfill $\Box$
\end{lemma}

\setcounter{section}{5}

\section{Self bumping and local connectivity}
\label{secselb}

By constructing appropriate paths in ${\cal GA}(\partial M)$ we  study the local connectivity of the set of doubly incompressible gallimaufries.  We note that by \cite{espoir}, the set of doubly incompressible gallimaufries is an open subset of ${\cal GA}(\partial M)$.

A note about the notation in this Section: The approximating gallimaufries all have the same moderate surface, namely $\partial_{\chi <0} (\partial M)$, and the same immoderate lamination, namely the empty lamination.  Therefore, we identify each approximating gallimaufry with the metric on $\partial_{\chi <0} (\partial M)$.  In particular, we speak of sequences of metrics on the moderate surface as converging to a limiting gallimaufry.

\begin{proposition}             \label{segments}
Let $M$ be a compact, orientable, hyperbolizable $3$-manifold with incompressible boundary. Let $\Gamma=(F,m,L)$ be a doubly incompressible gallimaufry and let $\{(\partial_{\chi<0} M,m_n,\emptyset)\}$ and $\{(\partial_{\chi<0} M,s_n,\emptyset)\}$ be two sequences ${\rm int}({\cal GA}(\partial M))$ converging to $\Gamma$. Then, up to passing to a subsequence, there is an arc $k_n\subset {\rm int}({\cal GA}(\partial M))$ joining $m_n$ to $s_n$ such that every sequence of points $\{z_n\in k_n\}$ converges to $\Gamma$.
\end{proposition}

\begin{proof}
The main difficulty is that on $F$ we need to have precise control on the behavior of $z_n$, whereas on $S(L)$ we need to control the large scale behavior of $z_n$. To deal with that issue, we use Fenchel-Nielsen coordinates to control the behavior of $z_n$ on $F$ and the measured geodesic lamination ${\cal F}_P(z_n)$ as defined in Section \ref{withboundary} to control the large scale behavior of $z_n$ on $S(L)$.

 We denote by $G$ the surface obtained by adding to $\overline{F}$ every component of $\partial M-\overline{F}$ that is an essential annulus. For each $n$ and $t$ we will define the metric $k_n(t)$ independently on $G$ and on the closure $H$ of $\partial_{\chi<0} M-G$.

Choose $\eps>0$ and consider a pants decomposition $P$ of $\partial M$ so that every component of $\partial\overline{F}$ is homotopic to a component of $P$ and so that $\ell_{m_n}(c)\geq\eps$ and $\ell_{s_n}(c)\geq\eps$ for every $n$ and every component $c$ of $P$ that is non-peripheral in $\overline{F}$.

Let us first define the restriction of $k_n(t)$ to $H$. In $(\partial_{\chi<0} M,m_n,\emptyset)$ and $(\partial_{\chi<0} M,s_n,\emptyset)$ we view $H$ as a subsurface with geodesic boundary. Since $\{ s_n\}$ and $\{ m_n\}$ tend to $\Gamma$,  we have $\ell_{m_n}(\partial H)\longrightarrow 0$ and $\ell_{s_n}(\partial H)\longrightarrow 0$. Then for $n$ large enough, there is associated to the restriction of $m_n$, respectively $s_n$, to $H$ a measured geodesic lamination ${\cal F}_P(m_n)$, respectively ${\cal F}_P(s_n)$ (see Section \ref{withboundary}). Furthermore the restriction of $m_n$ to $H$ is uniquely defined by ${\cal F}_P(m_n)$ and $\{\ell_{m_n}(c)| c\mbox{ is a component of } \partial H\}$.

Extract a subsequence such that $\{{\cal F}_P(m_n)\}$ converges to a geodesic lamination $R$ in the Hausdorff topology. Since the restriction of $m_n$ to $S(L)$ tends to $L$, then any subsequence of ${\cal F}_P(m_n)$ that converges projectively has a projective limit supported by $L$. It follows that $L$ is a sublamination of $R$. Consider a sequence $\{ \tau_n\}$ of train tracks carrying $R$ so that $\{\tau_n\}$ is a basis of neighborhoods of $R$ for the Hausdorff topology. Namely every sequence of laminations $\{R_n\}$ such that $R_n$ is minimally carried by $\tau_n$ converges to $R$ in the Hausdorff topology (for the existence of such train tracks, see \cite{conti}). Notice that if $R_n$ is minimally carried by $\tau_n$ then up to extracting a subsequence $\{R_n\}$ converges to a sublamination of $R$. Since $L$ is a sublamination of $R$, we may choose the $\tau_n$ so that $L$ is minimally carried by a sub-track of $\tau_n$ for every $n$.

 Since $\{ {\cal F}_P(m_n)\}$ converges to $R$, up to changing the indices, we may assume that ${\cal F}_P(m_n)$ is carried by $\tau_n$. In particular ${\cal F}_P(m_n)$ defines a weight system ${\cal F}_P(m_n)(\tau_n)$ on $\tau_n$. Let $\lambda$ be a measured geodesic lamination with support $L$, and note that ${\cal F}_P(m_n)$ also defines a weight system $\lambda(\tau_n)$ on $\tau_n$.  Let $K_n\rightarrow\infty$ be a sequence of positive real numbers. For $t\in[0,\frac{1}{2}]$, the weight system $2(\frac{1}{2}-t) {\cal F}_P(m_n)(\tau_n)+2 tK_n\lambda(\tau_n)$ defines a measured geodesic lamination $\mu_n(t)$ carried by $\tau_n$ such that $\mu_n(0)={\cal F}_P(m_n)$ and $\mu_n(\frac{1}{2})=K_n\lambda$. For $t\neq\frac{1}{2}$, $\mu_n$ is minimally carried by $\tau_n$. It follows that we have:  

\begin{claim}    \label{hausdlim}
Consider a sequence $\{t_n\}\in[0,\frac{1}{2}]^\N$. Then the Hausdorff limit of every subsequence of $\{\mu_n(t_n)\}$ is either $R$ or $L$. In particular it contains $L$.
\end{claim}

Consider a simple closed curve $c\subset G$.  The intersection number $i(c,\mu_n(t))$ is given by the following formula:

\begin{claim}    \label{interloin}
For $n$ large enough, depending on $c$, we have $i(c,\mu_n(t))=2(\frac{1}{2}-t)i(c,{\cal F}_P(m_n))+2t K_ni(c,\lambda)$ for all $t\in[0,1]$.
\end{claim}

\begin{proof}
Since $\{ \tau_n\}$ is a basis of neighborhoods of $R$ for the Hausdorff topology, for $n$ large enough, depending on $c$, we can change $\tau_n$ and $c$ by homotopies so that for every measured geodesic lamination $\lambda$ carried by $\tau_n$, $i(c,\lambda)$ is the sum of the weights of the branches of $\tau_n$ that $c$ intersects transversely, counted with multiplicity. In particular, for $n$ large enough, we have $i(c,\mu_n(t))=2(\frac{1}{2}-t)i(c,{\cal F}_P(m_n))+2t K_ni(c,\lambda)$.
\end{proof}

Using a similar construction, replacing $m_n$ by $s_n$ we get a path of measured geodesic laminations $\mu_n(t)$, $t\in [\frac{1}{2},1]$ with $\mu_n(\frac{1}{2})=K_n\lambda$ and $\mu_n(1)={\cal F}_P(s_n)$.  Hence, given a sequence $\{t_n\}\in[\frac{1}{2},1]^n$, the Hausdorff limit of  every subsequence of $\{\mu_n(t_n)\}$ contains $L$.

 We define the restriction of $k_n(t)$ to $H$ for $t\in[0,1]$ as follows: $\ell_{k_n(t)}(c)=(1-t)\ell_{m_n}(c)+t\ell_{s_n}(c)$ for every component $c$ of $\partial H$ and ${\cal F}_P(k_n(t))=\mu_n(t)$.

Next we define the restriction of $k_n(t)$ to $G$. Consider the Fenchel-Nielsen coordinates $\{(\ell_c,\theta_c)| c \mbox{ a component of } P\}$ associated to $P$. Namely for a metric $m$, $\ell_c(m)\in]0,\infty[$ is the $m$-length of $c$ and $\theta_c\in\R$ is the twist angle along $c$ (taking the convention that $\theta_c(m')=\theta_c(m)+2\pi$ corresponds to a full Dehn twist along $c$). On $G$ we define the restriction of $k_n(t)$ as follows: each component $c$ of $\partial G$ is a closed geodesic with length $\ell_c(k_n(t))=(1-t)\ell_c(m_n)+t\ell_c(s_n)$; the Fenchel-Nielsen coordinates of $k_n(t)$ along each non-peripheral component $c$ of $P\cap G$ are $\{(1-t)\ell_c(m_n)+t\ell_c(s_n),(1-t)\theta_c(m_n)+t\theta_c(s_n)\}$. This defines a metric with geodesic boundary on $G$.

Notice that a component $d$ of $\partial G$ is also a component of $\partial H$ and that the length given to $d$ by the restriction of $k_n(t)$ to $G$ is the same as the length given by the restriction to $H$. We just need to define how we glue together the metrics defined on $G$ and $H$ to get a complete definition of $k_n(t)$. For a component $c$ of $\partial G$, the Fenchel-Nielsen coordinates of $k_n(t)$ along $c$ are $\{(1-t)\ell_c(m_n)+t\ell_c(s_n),(1-t)\theta_c(m_n)+t\theta_c(s_n)\}$.

Now a complete hyperbolic metric $k_n(t)$ has been defined on $\partial_{\chi<0} M$ for every $n\in\N$ and all $t\in[0,1]$. It follows easily from the definition that, for a fixed $n$, $k_n(t)$  is a continuous path joining $m_n$ to $s_n$. The following two Claims will conclude the proof of Proposition \ref{segments}.

\begin{claim}           \label{tiersborne}
For every sequence $\{t_n\}\subset [0,1]^\N$ the restrictions of the $k_n(t_n)$ to $G$ tend to the complete hyperbolic metric $m$ on $F$.
\end{claim}

\begin{proof}
Since $\{ m_n\}$ and $\{ s_n\}$ both converge to $m$ on $F$, for every component $c$ of $P$ that is homotopic to a cusp of $F$, we have $\ell_{m_n}(c)\longrightarrow 0$ and $\ell_{s_n}(c)\longrightarrow 0$. Thus we get $\ell_{k_n(t_n)}(c)=t_n\ell_c(m_n)+(1-t_n)\ell_c(s_n)\longrightarrow 0$. If $c$ is a component of $P\cap G$ that is not homotopic to a cusp of $F$, then  $\{\ell_c(m_n),\theta_c(m_n)\}$ and $\{\ell_c(s_n),\theta_c(s_n)\}$ both converge to $\{\ell_c(m),\theta_c(m)\}$. Hence $\{\ell_c(k_n(t_n)),\theta_c(k_n(t_n))\}=\{t_n\ell_c(m_n)+(1-t_n)\ell_c(s_n),t_n\theta_c(m_n)+(1-t_n)\theta_c(s_n)\}$ converges to $\{\ell_c(m),\theta_c(m)\}$. It follows that the restrictions of the $k_n(t)$ to $G$ tend to the complete hyperbolic metric $m$ on $F$.
\end{proof}

\begin{claim}   \label{tiersunborne}
Let $L_1$ be a component of $L$. Then for every sequence $\{t_n\}\in [0,1]^\N$, the restrictions of the $k_n(t_n)$ to $S(L)$ tend to $L$.
\end{claim}

\begin{proof}
We will assume that we have $\{t_n\}\in [0,\frac{1}{2}]^\N$. The proof is similar if we assume $\{t_n\}\in [\frac{1}{2},1]^\N$ and these two results are enough to conclude the proof of Proposition \ref{segments}.
Consider a non-peripheral simple closed curve $c\subset S(L_1)$. Since the restrictions of the $m_n$ to $S(L_1)$ tend to $L_1$, we have $i({\cal F}_P(m_n),c)\longrightarrow\infty$ and $\frac{i({\cal F}_P(m_n),\partial\overline{S(L_1)})}{i({\cal F}_P(m_n),c)}\longrightarrow 0$. We also have that $K_n i(\lambda,c)\longrightarrow\infty$ and $K_n i(\lambda,d)=0$, where $d$ is as defined in the second paragraph before the statement of Claim \ref{tiersborne}. From Claim \ref{interloin} and Theorem \ref{compactlam}, we get $\ell_{k_n(t_n)}(c)\longrightarrow\infty$ and $\frac{\ell_{k_n(t_n)}(\partial\overline{S(L_1)})}{\ell_{k_n(t_n)}(c)}\longrightarrow\infty$. Hence, up to extracting a subsequence the restrictions of the $k_n(t_n)$ to $S(L_1)$ tend to a projective measured geodesic lamination. By Claim \ref{hausdlim}, the support of this projective lamination is $L_1$. This concludes the proof of Claim \ref{tiersborne}.
\end{proof}

By Claims \ref{tiersborne} and \ref{tiersunborne}, for every sequence $\{t_n\}\in [0,1]^\N$, $\{k_n(t_n)\}$ converges to $\Gamma$.
\end{proof}

Using Theorem \ref{conf} and Proposition \ref{segments}, we can now prove Theorems \ref{locon} and \ref{selfb}.

\setcounter{theorem}{0}

\begin{theorem}
Let $M$ be a compact, orientable, hyperbolizable $3$-manifold with incompressible boundary. Then the space ${\rm SH}(\pi_1(M))$ is locally connected.
\end{theorem}

\begin{proof}
Consider a representation $\rho\in {\rm SH}(\pi_1(M))$.  By the tameness of $\rho(\pi_1(M))$, see \cite{bouts}, there is a compact manifold $M'$ such that $\rho$ uniformizes $M'$. By Lemma \ref{topo}, every sufficiently small neighborhood ${\cal V}\subset {\rm SH}(\pi_1(M))$ of $\rho$ lies in ${\rm SH}(M')$. Let $\Gamma\in{\cal GA}(\partial M')$ be the end invariants of $\rho$. By Theorem \ref{conf}, ${\cal V}$ is homeomorphic to a neighborhood ${\cal W}\subset {\cal GA}(\partial M')$ of $\Gamma$. By Proposition \ref{segments}, ${\cal W}$ contains a connected neighborhood ${\cal W}'$ of $\Gamma$. Taking the preimage of ${\cal W}'$ under the ending map, we get a connected neighborhood ${\cal V}'\subset{\cal V}$ of $\rho$. Thus we have proved that ${\rm SH}(\pi_1(M))$ is locally connected.
\end{proof}

The proof of Theorem \ref{selfb} follows the same lines.

\setcounter{theorem}{1}

\begin{theorem}
Let $M$ be a compact, orientable, hyperbolizable $3$-manifold with incompressible boundary.
Let $\rho\in {\rm SH}(\pi_1(M))$ be a representation uniformizing $M$. Then every neighborhood of $\rho$ contains a neighborhood ${\cal V}\subset {\rm SH}(\pi_1(M))$ of $\rho$ such that ${\cal V}\cap {\rm int}({\rm SH}(\pi_1(M)))$ is connected.
\end{theorem}

\begin{proof}
Let ${\cal V}_1\subset {\rm SH}(\pi_1(M))$ be a neighborhood of $\rho$. By Lemma \ref{topo}, ${\cal V}_1$ contains a neighborhood ${\cal V}_2\subset {\rm SH}(M)$ of $\rho$. We denote by $e$ the ending map as defined in Theorem \ref{conf}. By Proposition \ref{segments}, $e({\cal V}_2)$ contains a neighborhood ${\cal W}\subset {\rm Im}(e)$ of $e(\rho)$ such that ${\cal W}\cap {\rm int}({\cal GA}(\partial M))$ is connected. Taking ${\cal V}=e^{-1}({\cal W})$, we get a neighborhood of $\rho$ such that ${\cal V}\cap {\rm int}({\rm SH}(\pi_1(M)))$ is connected.
\end{proof}

\section{The action of ${\rm Mod}(M)$}
\label{secmodul}

In this last Section we  study the action of ${\rm Mod}(M)$ on ${\rm SH}(M)$.

\setcounter{theorem}{2}

\begin{theorem}
Let $M$ be a compact, orientable, hyperbolizable $3$-manifold with incompressible boundary. Assume that $M$ is not an $I$-bundle over a closed surface. Then the action of ${\rm Mod}(M)$ on ${\rm SH}(M)$ is properly discontinuous. 
\end{theorem}

\begin{proof}
Consider a sequence of representations $\{\rho_n\}\subset {\rm int}({\rm SH}(M))$ and a sequence of diffeomorphisms $\{\phi_n: M\rightarrow M\}$ such that $\{\rho_n\}$ and $\{\rho_n\circ\phi_{n*}\}$ converge respectively to representations $\rho_\infty$  and $\rho'_\infty$ (up to conjugacy) in ${\rm SH}(M)$. We  show that up to extracting a subsequence, the $\phi_n$ are isotopic. It follows easily that the action of ${\rm Mod}(M)$ on ${\rm SH}(M)$ is properly discontinuous. Since minimally parabolic geometrically finite representations are dense in ${\rm SH}(M)$, we may assume that $\rho_n$ is such a representation without any loss of generality. As a first step, the following Claim  is useful to pick the correct representation in each conjugacy class.

\begin{claim}   \label{commute}
Let $\{\gamma_n\}$, $\{a_n\}$ and $\{b_n\}$ be sequences of elements of ${\rm PSL}(2,\C)$. Assume that the sequences $\{a_n\}$ and $\{b_n\}$ converge to hyperbolic isometries $a_\infty$ and $b_\infty$ such that $a_\infty$ and $b_\infty$ do not have a common fixed point. Assume also that the sequences $\{\gamma_n a_n\gamma_n^{-1}\}$ and $\{\gamma_n b_n\gamma_n^{-1}\}$ converge. Then, up to extracting a subsequence, $\{\gamma_n\}$ converges.
\end{claim}

\begin{proof}
Assume that no subsequence of $\{\gamma_n\}$ converges.  We show that we end up with a contradiction. For $n$ sufficiently large, $a_n$ and $b_n$ are hyperbolic isometries. Let $A_n$ and $B_n$ be the axes of $a_n$ and $b_n$ respectively, for $n\in\N\cup\{\infty\}$. By assumption $\{A_n\}$, respectively $\{B_n\}$, converges to $A_\infty$, respectively $B_\infty$, in the Hausdorff topology on $\Hp^3\cup\partial_\infty \Hp^3$. Consider a point $x\in\Hp^3$. Since we have assumed that no subsequence of $\{\gamma_n\}$ converges, we have ${\rm d}(x,\gamma_n^{-1}(x))\longrightarrow\infty$. Since $a_\infty$ and $b_\infty$ do not have a common fixed point, we have $\max\{{\rm d}(\gamma_n^{-1}(x),A_n),{\rm d}(\gamma_n^{-1}(x),B_n)\}\longrightarrow\infty$.

 By assumption the sequence $\{{\rm d}(x,\gamma_n a_n\gamma_n^{-1}(x))\}$ converges. On the other hand, we have ${\rm d}(x,\gamma_n a_n\gamma_n^{-1}(x))={\rm d}(\gamma_n^{-1}(x), a_n\gamma_n^{-1}(x))$ which converges if and only if $\{{\rm d}(\gamma_n^{-1}(x),A_n)\}$ is a bounded sequence. Similarly we get that $\{{\rm d}(\gamma_n^{-1}(x),B_n)\}$ is a bounded sequence. Thus we get that $\max\{{\rm d}(\gamma_n^{-1}(x),A_n),{\rm d}(\gamma_n^{-1}(x),B_n)\}$ is bounded, contradicting the preceding paragraph.
\end{proof}

This Claim allows us to get the expected conclusion under some extra assumptions on the diffeomorphisms $\phi_n$.

\begin{lemma}           \label{extr}
Let $M$ be a compact, orientable, hyperbolizable $3$-manifold with incompressible boundary.  Let $N\subset\partial M$ be an incompressible compact submanifold such that $\pi_1(N)$ is not Abelian. Let $\{\phi_n:M\rightarrow M\}$ be a sequence of diffeomorphisms such that, up to homotopy, $\phi_{n} |_N$ is the identity and let $\{\rho_n\}\subset {\rm int}({\rm SH}(M))$ be a convergent sequence. Assume that the sequence $\{\rho_n\circ\phi_{n*}\}$ converges as well. Then there is a diffeomorphism $\phi:M\rightarrow M$ such that up to extracting a subsequence, each $\phi\circ\phi_n$ is isotopic to the identity. 
\end{lemma}

\begin{proof}
Let $\rho_n:\pi_1(M)\rightarrow {\rm PSL}(2,\C)$ be a representative of $\rho_n\in {\rm int}({\rm SH}(M))$ such that $\{\rho_n\}$ converges to $\rho_\infty$, namely we have actual convergence and not only just up to conjugacy. Consider the sequence $\{\rho_n\circ\phi_{n*}:\pi_1(M)\rightarrow {\rm PSL}(2,\C)\}$. By assumption there is a sequence $\{\gamma_n\}\subset {\rm PSL}(2,\C)$ so that $\{\gamma_n(\rho_n\circ\phi_{n*})\gamma_n^{-1}\}$ converges. Since $\pi_1(N)$ is not Abelian, the non-elementary Kleinian group $\rho_\infty(\pi_1(N))$ contains two hyperbolic isometries $\rho_\infty(g_1)$ and $\rho_\infty(g_2)$ with disjoint fixed point sets. By assumption, we have $\rho_n\circ\phi_{n*}(g_i)=\rho_n(g_i)$ for $i=1,2$. Thus taking $a_n=\rho_n(g_1)$ and $b_n=\rho_n(g_2)$, we get from Claim \ref{commute} that, up to extracting a subsequence, $\{\gamma_n\}$ converges. Hence, up to extracting a subsequence $\{\rho_n\circ\phi_{n*}\}$ converges. 

We have $\rho_n\circ\phi_{n*}(\pi_1(M))=\rho_n(\pi_1(M))$, as follows. Since $\{\rho_n\circ\phi_{n*}\}$ converges and since $\{\rho_n\}$ converges strongly, for every $a\in\pi_1(M)$ there is $b\in\pi_1(M)$ such that  $\{\rho_n\circ\phi_{n*}(a)\}$ converges to $\rho_\infty(b)$. It follows from \cite{jorgi} that for $n$ sufficiently large $\phi_{n*}(a)=b$. Thus we have proven that for every closed curve $a\subset M$ there is a closed curve $b\subset M$ such that $\phi_n(a)$ is homotopic to $b$ for $n$ sufficiently large.

Consider a pants decomposition $P$ of $\partial_{\chi<0}M$. By the above paragraph, for $n$ sufficiently large, $\phi_n$ maps each component $c$ of $P$ to some given curve $d\subset\partial M$ that is freely homotopic in $M$ to $c$. By \cite{jojo} and \cite{jacs}, for a given curve $c$, there are only finitely many such simple closed curves $d$ up to isotopy on $\partial M$. It follows that up to extracting a subsequence there is $n_0$ such that $\phi_{n_0}^{-1}\circ\phi_n$ maps $P$ to itself up to isotopy. Consider another pants decomposition $P'\subset\partial_{\chi<0} M$ such that the components of $\partial_{\chi<0} M-(P\cup P')$ are discs. Using the same arguments we find a diffeomorphism $\phi$ and a subsequence of $\{\phi_n\}$ such that $\phi\circ\phi_n (P)$ is isotopic to $P$ and $ \phi\circ\phi_n(P')$ is isotopic to $P'$ in $\partial M$. It follows that the restrictions of the $\phi\circ\phi_n$ to $\partial_{\chi<0} M$ are isotopic to the identity. By \cite{jojo}, up to passing to a subsequence,  $\phi\circ\phi_n$ is isotopic to the identity for every $n$. This concludes the proof of Lemma \ref{extr}.
\end{proof}

This is enough to conclude the proof of Theorem \ref{modul}  in most cases that remain. Let $W$ be a characteristic submanifold for $M$. Such a characteristic submanifold is an union of $I$-bundles and solid tori and its basic property is that every essential annulus in $M$ is isotopic to an annulus lying in $W$. The existence of such a characteristic submanifold has been proved in \cite{jojo} and \cite{jacs}. If we assume that a component $N$ of $M-W$ has a non-Abelian fundamental group, then we can argue as follows. By \cite{jojo}, up to extracting a subsequence there is a diffeomorphism $\phi:M\rightarrow M$ such that, up to homotopy, each $(\phi\circ\phi_n) |_{M-W}$ is the identity. In particular, up to homotopy, $(\phi\circ\phi_n) |_{N}$ is the identity. By assumption the sequences $\{\rho_n\}$ and $\{\rho_n\circ(\phi\circ\phi_n)_*\}$ converge in ${\rm SH}(M)$. It follows from Lemma \ref{extr} that up to extracting a subsequence and up to changing $\phi$, the diffeomorphism $\phi\circ\phi_n$ is isotopic to the identity for every $n$. Hence, up to extracting a subsequence, the diffeomorphisms $\phi_n$ are isotopic and we are done.

When every component of $M-W$ has Abelian fundamental group, the closure of $M-W$ is an union of solid tori. In this case we will use Theorem \ref{conf} to find an incompressible surface $H\subset\partial M$ satisfying the hypotheses of Lemma \ref{extr}.
Let $\Gamma_n=(F_n,m_n,L_n)$ and $\Gamma'_n=(F'_n,m'_n,L'_n)$ be the ending gallimaufries of $\rho_n$ and $\rho'_n=\rho_n\circ\phi_{n*}$, respectively. Since we have assumed that $\rho_n$ is minimally parabolic and geometrically finite, we have $F_n=F'_n=\partial_{\chi<0} M$ and $L_n=L'_n=\emptyset$. Let $\Gamma_\infty=(F_\infty,m_\infty,L_\infty)$ and $\Gamma'_\infty=(F'_\infty,m'_\infty,L'_\infty)$ be the ending gallimaufries of $\rho_\infty$ and $\rho'_\infty$, respectively. By Theorem \ref{conf}, $\{\Gamma_n\}$ converges to $\Gamma_\infty$ and $\{\Gamma'_n\}$ converges to $\Gamma'_\infty$. To each end of $\Hp^3/\rho_\infty(\pi_1(M))$ corresponds a subsurface of $\partial M$; we will say that this subsurface {\it faces} an end of $\Hp^3/\rho_\infty(\pi_1(M))$. Now we  reduce the search for $H$ to the search for a curve $c\subset\partial M$ with some specific properties.

\begin{claim}     \label{isof}
Let $c\subset\partial M$ be a simple closed curve such that $\phi_n(c)$ is isotopic to $c$ on $\partial M$ for every $n$. Assume that there are surfaces $H\supset c$ and $H'\supset c$ such that $H$ faces an end of $\Hp^3/\rho_\infty(\pi_1(M))$ and that $H'$ faces an end of $\Hp^3/\rho'_\infty(\pi_1(M))$. Assume that $c$ is peripheral in neither  $H$ nor  $H'$. Then there exist a diffeomorphism $\phi:M\rightarrow M$ and a subsequence such that each $(\phi\circ\phi_n) |_{H}$ is isotopic to the identity.
\end{claim}

\begin{proof}
First, we show that the ends faced by $H$ and $H'$ are of the same type. 

\begin{claim}          \label{sametype}
The surface $H'$ faces a geometrically finite end of $M'_\infty$ if and only if $H$ faces a geometrically finite end of $M_\infty$.
\end{claim}
\begin{proof}
By Theorem \ref{conf}, $H$, respectively $H'$, faces a geometrically finite end if and only if $\{ \ell_{m_n}(c)\}$, respectively $\{ \ell_{m_n}(\phi_n(c))\}$, is bounded. By assumption, $\phi_n(c)$ is isotopic to $c$ hence we have $\ell_{m_n}(\phi_n(c))=\ell_{m_n}(c)$. Thus $H'$ faces a geometrically finite end of $M'_\infty$ if and only if $H$ faces a geometrically finite end of $M_\infty$.
\end{proof}

Consider the case where $H$ and $H'$ face geometrically finite ends of $M_\infty$ and $M'_\infty$ respectively. By Theorem \ref{conf}, $\{m_n |_H\}$ and $\{m'_n |_{H'}\}$ converge to complete hyperbolic metrics on $H$ and $H'$ respectively. Since $\phi_n(H)$ contains $c$, $\phi_n(H)$ intersects $H'$. On the other hand,  $\ell_{m'_n}(\phi_n(\partial\overline{H}))\longrightarrow 0$. Thus, for $n$ large enough, $\phi_n(\partial\overline{H})\cap\overline{H'}\subset\partial\overline{H'}$ up to isotopy. It follows that up to isotopy, for $n$ large enough, $\phi_n(H)$ either is disjoint from $H'$ or contains $H'$. Since $c\subset\phi_n(H)$ (up to isotopy) and $c\subset H'$, $H'\subset \phi_n(H)$. Similarly, $\phi_n^{-1}(\partial\overline{H'})\cap\overline{H}\subset\partial\overline{H}$ up to isotopy. It follows that $H\subset \phi_n^{-1}(H')$ for $n$ large enough. Thus we have proved that $\phi_n(H)=H'$, up to isotopy, for $n$ large enough. Up to extracting a subsequence and composing by a fixed diffeomorphism, we may assume $H'=H$ up to isotopy.

Given a curve $d\subset H$, there is $K$ such that $\ell_{m_n}(d)\leq K$. We also have $\ell_{m'_n}(\phi_n(d))=\ell_{m_n}(d)\leq K$. Since $\Gamma'_n$ converges to $\Gamma_\infty$, $m'_n$ converges to $m'_\infty$. It follows that for $n$ large enough, we have $\ell_{m'_\infty}(\phi_n(d))\leq 2K$. There are only fintely many $m'_\infty$-geodesics with length bounded by $K$. Hence there are only finitely many possibilities for the isotopy class of $\phi_n(d)$. Since this holds for any closed curve $d\subset H$, there are only finitely many possibilities for the isotopy class of the diffeomorphism $\phi_n |_H:H\rightarrow H$. This concludes the proof of Claim \ref{isof} when $H$ faces a geometrically finite end of $M_\infty$.\\

\indent
Assume now that $H$ faces a geometrically infinite end of $M_\infty$ with ending lamination $L$ and that $H'$ faces a geometrically infinite end of $M'_\infty$ with ending lamination $L'$. Consider pants decompositions $P$ and $P'$ of $\partial_{\chi<0} M$ such that $\ell_{m_n}(d)\leq\eps$ for every component $d$ of $P$ and $\ell_{m'_n}(d')\leq\eps$ for every component $d'$ of $P'$. Consider $\eps_n\longrightarrow 0$ such that, up to extracting a subsequence, $\eps_n{\cal F}_P(m_n)\cap H$ tends to a measured geodesic lamination $\lambda$. Since $\Gamma_n$ tend to $\Gamma_\infty$, the support of $\lambda$ is the ending lamination $L_\infty\cap H$ of $H$, which in particular it fills $H$. By Claim \ref{sametype}, the restriction of the $m'_n$ to $H'$ does not contain a convergent subsequence. Extract a subsequence such that $\{ \eps_n{\cal F}_{P'}(m'_n)\cap H'\}$ tends to a measured geodesic lamination $\lambda'$. Since $\{ \Gamma'_n\}$ converges to $\Gamma'_\infty$, the support of $\lambda'$ is the support of the ending lamination $L'_\infty\cap H'$ of $H'$, which fills $H'$.

If $\partial\overline{H}$ is empty, then $H$ is a component of $\partial M$ and the same holds for $H'$. Since $c\subset \phi_n(H)$ and $c\subset H'$ by assumption, we have $H'=\phi_n(H)$. Otherwise, for a component $d$ of $\partial\overline{H}$, we have $\eps_n\ell_{m_n}(d)\longrightarrow 0$. Hence $\eps_n\ell_{m'_n}(\phi_n(d))\longrightarrow 0$. Extract a subsequence such that $\{ \phi_n(d)\}$ converges in the Hausdorff topology to a geodesic lamination $D$. Since $d$ is disjoint from $c$, $D$ is disjoint from $c$. Since $|\lambda'|$ fills $H'$, if $D$ were to intersect $H'$ without being peripheral, then $D$ would cross $|\lambda'|$. Since $\{ \eps_n{\cal F}_{P'}(m'_n)\cap H'\}$ tends to $\lambda'$ we would then have $\liminf\eps_n\ell_{m'_n}(\phi_n(d))\geq\liminf\eps_n i({\cal F}_{P'}(m'_n),\phi_n(d))>0$. Thus we have proved $\phi_n(\partial\overline{H})\cap\overline{H'}\subset\partial\overline{H'}$ up to isotopy. Similarly we have $\phi_n^{-1}(\partial\overline{H'})\cap\overline{H}\subset\partial\overline{H}$ up to isotopy and we may assume  $H'=\phi_n(H)$ up to isotopy, for the same reasons as in the geometrically finite case above.

Let us extend $c$ to a pants decomposition $C$ of $H$ and consider a leaf $d$ of $C$. Extract a subsequence such that $\{ \phi_n(d)\}$ converges in the Hausdorff topology to a geodesic lamination $D\subset H$. Since $d$ is disjoint from $c$, $D$ is disjoint from $c$ and hence crosses $|\lambda|$. If $D$ is not a simple closed curve, we have $\eps_n i({\cal F}_{P'}(m'_n),\phi_n(d))\longrightarrow\infty$. By Theorem \ref{compactlam}, we have $\eps_n\ell_{m'_n}(\phi_n(d))\longrightarrow\infty$. This would contradict $\eps_n\ell_{m'_n}(\phi_n(d))=\eps_n\ell_{m_n}(d)\longrightarrow i(d,\lambda)$. Thus, up to extracting a subsequence $\phi_n(d)$ does not depend on $n$. It follows that there is $\phi:M\rightarrow M$ such that up to extracting a subsequence $\{ \phi\circ\phi_n\}$ leaves $C$ invariant, up to isotopy. For each leaf $d$ of $C$, we choose a transverse $t$, namely a simple closed curve that crosses $d$ and is disjoint from $C-d$. By the same argument, up to extracting a subsequence $\phi_n(t)$ does not depend on $t$. It follows that, up to changing $\phi$, $(\phi\circ\phi_n)|_H$ is isotopic to the identity.
\end{proof}

Now we can conclude the proof of Theorem \ref{modul} by finding a curve $c$ satisfying the assumption of Claim \ref{isof}. Let $W\subset M$ be a characteristic submanifold.  By \cite{jojo}, up to extracting a subsequence, there is a diffeomorphism $\phi:M\rightarrow M$ such that $(\phi\circ\phi_n) |_{M-W}$ is isotopic to the identity. As we have seen after the proof of Lemma \ref{extr}, if a component of $M-W$ has a non-Abelian fundamental group, we are done.

Assume that all components of $M-W$ have Abelian fundamental groups. Then all the components of the closure of $M-W$ are solid tori and either $M$ is an $I$-bundle or decomposes as the union of $I$-bundles over compact surfaces with boundary, $I$-bundles over tori and solid tori. 

If a component $T_1$ of this decomposition is an $I$-bundle over a torus, $\partial T_1\cap\partial M$ is the union of one torus and some (at least one) annuli and $\phi_n$ exchanges the components of $T_1\cap\partial M$ (up to isotopy). Hence there is a homeomorphism $\psi:M\rightarrow M$ such that $\psi\circ\phi_n$ maps each component of $T_1\cap\partial M$ to itself (up to isotopy and) up to extracting a subsequence. Since $\Gamma_\infty$ and $\Gamma'_\infty$ are doubly incompressible, every curve $c$ lying in such an annulus lies in the "middle" of an end. Namely for every representation $\rho\in {\rm SH}(M)$, $c$ lies in a surface $H$ facing an end of $M_\rho$ and $c$ is not peripheral in $\overline{H}$. Since $\psi\circ\phi_n$ maps each component of $T_1\cap\partial M$ to itself, $\psi\circ\phi_n(c)$ is isotopic to $c$ on $\partial M$. Hence $c$ fulfills the hypothesis of Claim \ref{isof} (replacing $\phi_n$ with $\psi\circ\phi_n$). This provides us with an incompressible surface $H\subset\partial M$ with non-Abelian fundamental group  and a diffeomorphism $\phi:m\rightarrow M$ such that $(\phi\circ\phi_n)_{|H}$ is isotopic to the identity.

If a component $T_1$ of  the decomposition of $M$ is a solid torus, $T_1\cap\partial M$ is the union of annuli. If those annuli are not primitive then we are in the same situation as before. Namely, every curve $c$ lying in such an annulus lies in the "middle" of an end. Hence we can conclude as in the preceding paragraph.

We are left with the case where $M$ is the union of $I$-bundles over surfaces with boundary and solid tori and where for each such solid torus $T$, $T\cap\partial M$ is an union of primitive annuli. By assumption $M$ is not an $I$-bundle.  This is possible only if at least one component $T_1$ of $W$ is a solid torus such that $T_1\cap\partial M$ has at least three components. Furthermore $\phi_n$ exchanges the components of $T_1\cap\partial M$, hence there is $\psi$ such that $\psi\circ\phi_n$ maps each component of $T_1\cap\partial M$ to itself. Since $\Gamma_\infty$ and $\Gamma'_\infty$ are doubly incompressible, at most one simple closed curve $d\subset T_1\cap\partial M$, respectively  $d'\subset T_1\cap\partial M$, is peripheral in a surface $G$ facing an end of $M_\infty$, respectively  in a surface $G'$ facing an end of $M'_\infty$. Since $T_1\cap\partial M$ contains at least three non-isotopic simple closed curves, we can chose a simple closed curve $c\neq d$ with $\phi_n(c)\neq d'$ and $\psi\circ\phi_n(d)$ isotopic to $d$ for infinitely many $n$. Thus $c$ lies in a surface $H$ facing an end of $M_\infty$, $c$ is not peripheral in $\overline{H}$, $\phi_n(c)$ lies in a surface $H'$ facing an end of $M'_\infty$ and $\phi_n(c)$ is not peripheral in $\overline{H'}$ and we can conclude as before.

The surface $H$ produced in Claim \ref{isof} provides us with an incompressible manifold $N=H$ satisfying the hypotheses of Lemma \ref{extr}. From Lemma \ref{extr}, we conclude that up to extracting a subsequence, the diffeomorphisms $\phi_n$ are isotopic. This concludes the proof of Theorem \ref{modul}.
\end{proof}

When $M$ is a trivial $I$-bundle over a closed surface $S$ it is easy to see that Theorem \ref{modul} does not hold. Consider a pseudo-Anosov diffeomorphism $\phi:S\rightarrow S$ with stable lamination $\lambda^+$ and unstable lamination $\lambda^-$. Consider the diffeomorphism $\phi:M\rightarrow M$ whose projection along the fibers is $\phi$. Consider the representation $\rho\in\partial {\rm SH}(M)$ which is doubly degenerate and has ending laminations $|\lambda^+|$ on one side and $|\lambda^-|$ on the other side. Then $\rho$ is a fixed point of the action of $\phi$ on ${\rm SH}(M)$.

\addcontentsline{toc}{chapter}{Bibliography}

\footnotesize{

}

\end{document}